\documentclass[11pt]{amsart}
\usepackage[dvipsnames]{xcolor}
\usepackage[left=1in,top=1in,right=1in,bottom=1in]{geometry}
\usepackage{wasysym, mathtools, amssymb, tikz, mathrsfs, bxeepic, pict2e}

\usepackage[T1]{fontenc}

\usepackage[colorlinks=true, linkcolor=red!80!black, urlcolor=purple, citecolor=blue!70!black]{hyperref}
\usetikzlibrary{matrix,arrows}

\input xy 
\xyoption{all} %

\newtheorem{theorem}{Theorem}[section]
\newtheorem{lemma}[theorem]{Lemma}
\newtheorem{corollary}[theorem]{Corollary}
\newtheorem{proposition}[theorem]{Proposition}
\newtheorem{conjecture}[theorem]{Conjecture}

\newtheorem{prop}[theorem]{Proposition}

\theoremstyle{definition}
\newtheorem{definition}[theorem]{Definition}
\newtheorem{example}[theorem]{Example}

\newtheorem{remark}[theorem]{Remark}
\newtheorem{Expectation}[theorem]{Expectation}

\newcommand{\Aff}{{\mathbb A}}
\newcommand{\C}{{\mathbb C}}
\newcommand{\bC}{{\mathbb C}}
\newcommand{\F}{{\mathbb F}}
\newcommand{\G}{{\mathbb G}}

\newcommand{\PP}{{\mathbb P}}
\newcommand{\bP}{{\mathbb P}}
\newcommand{\Q}{{\mathbb Q}}
\newcommand{\R}{{\mathbb R}}
\newcommand{\Z}{{\mathbb Z}}

\def\bbar#1{\setbox0=\hbox{$#1$}\dimen0=.2\ht0 \kern\dimen0 
\overline{\kern-\dimen0 #1}}
 
\newcommand{\kbar}{{\bbar{k}}}

\newcommand{\fp}{{\mathfrak p}}

\newcommand{\sF}{{\mathscr F}}

\newcommand{\sL}{{\mathscr L}}

\newcommand{\calA}{{\mathcal A}}
\newcommand{\calB}{{\mathcal B}}
\newcommand{\calC}{{\mathcal C}}
\newcommand{\tcalC}{\widetilde{\calC}}
\newcommand{\calD}{{\mathcal D}}
\newcommand{\calE}{{\mathcal E}}

\newcommand{\calI}{{\mathcal I}}

\newcommand{\calL}{{\mathcal L}}

\newcommand{\calO}{{\mathcal O}}

\newcommand{\calT}{{\mathcal T}}

\newcommand{\calX}{{\mathcal X}}

\newcommand{\logct}{\Omega_{X}^{1}(\log D)}

\DeclareMathOperator{\supp}{supp}

\DeclareMathOperator{\divv}{div}
\DeclareMathOperator{\ord}{ord}
\DeclareMathOperator{\Sym}{Sym}

\DeclareMathOperator{\Pic}{Pic}
\DeclareMathOperator{\Jac}{Jac}

\DeclareMathOperator{\Spec}{Spec}
\DeclareMathOperator{\Proj}{Proj}

\begin{document}
\title[Hyperbolicity of varieties of log general type]{Hyperbolicity of varieties of log general type}
\author{Kenneth Ascher}
\address{Department of Mathematics, Princeton University, Princeton, NJ 08544, USA}
\email{kascher@princeton.edu}
\author{Amos Turchet}
\address{
Centro di Ricerca Matematica Ennio De Giorgi, Scuola Normale Superiore, Palazzo Puteano, Piazza dei cavalieri, 3, 56100 Pisa, Italy}
\email{amos.turchet@sns.it}

\maketitle
\setcounter{tocdepth}{1}

\section{Introduction}\label{sec:intro}
Diophantine Geometry aims to describe the sets of rational and/or integral points on a variety. More precisely one would like \emph{geometric} conditions on a variety $X$ that determine the distribution of rational and/or integral points. Here geometric means conditions that can be checked on the algebraic closure of the field of definition.

Pairs, sometimes called log pairs, are objects of the form $(X,D)$ where $X$ is a projective variety and $D$ is a reduced divisor. These objects naturally arise in arithmetic when studying integral points, and play a central role in geometry, especially in the minimal model program and the study of moduli spaces of higher dimensional algebraic varieties. They arise naturally in the study of integral points since, if one wants to study integral points on a quasi-projective variety $V$, this can be achieved by studying points on $(X, D)$, where $(X \setminus D) \cong V$, the variety $X$ is a smooth projective compactification of $V$, and the complement $D$ is a normal crossings divisor. In this case, $(X,D)$ is referred to as a \emph{log pair}.

The goal of these notes is threefold:
\begin{enumerate}
    \item to present an introduction to the study of rational and integral points on curves and higher dimensional varieties and pairs;
    \item to introduce various notions of hyperbolicity for varieties and pairs, and discuss their conjectural relations; 
    \item and to show how geometry influences the arithmetic of algebraic varieties and pairs using tools from birational geometry.
\end{enumerate}

Roughly speaking, a $k$-rational point of an algebraic variety is a point whose coordinates belong to $k$. One of the celebrated results in Diophantine geometry of curves is the following. 
\begin{theorem}
If $\calC$ is a geometrically integral smooth projective curve over a number field $k$, then the following are equivalent.
\begin{enumerate}
    \item $g(\calC) \geq 2$, 
    \item the set of $L$-rational points is finite for every finite extension $L/k$ [Faltings' theorem; arithmetic hyperbolicity], 
    \item every holomorphic map $\bC \to \calC^{\mathrm{an}}_\bC$ is constant [Brody hyperbolicity], and
    \item the canonical bundle $\omega_\calC$ is ample. 
\end{enumerate}
\end{theorem}

In particular, one can view the above theorem as saying that various notions of hyperbolicity coincide for projective curves. One of the major open questions in this area is how the above generalizes to higher dimensions. The following conjecture we state is related to the Green-Griffiths-Lang conjecture. 

\begin{conjecture}\label{conj:intro}
Let $X$ be a projective geometrically integral variety over a number field $k$. Then, the following are equivalent:
\begin{enumerate}
    \item $X$ is arithmetically hyperbolic,
    \item $X_\bC$ is Brody hyperbolic, and 
    \item every integral subvariety of $X$ is of general type.
\end{enumerate}
\end{conjecture}

We recall that a variety is of general type if there exists a desingularization with big canonical bundle. This conjecture is very much related to conjectures of Bombieri, Lang, and Vojta postulating that varieties of general type (resp. log general type) do not have a dense set of rational (resp. integral) points. While Conjecture \ref{conj:intro} is essentially wide open, it is known that if the cotangent bundle $\Omega^1_X$ is sufficiently positive, all three conditions are satisfied. In particular, the latter two are satisfied if $\Omega^1_X$ is ample, and the first is satisfied if in addition $\Omega^1_X$ is globally generated. We do note, however, that there are examples of varieties that are (e.g. Brody) hyperbolic but for which $\Omega^1_X$ is not ample (see Example \ref{ex:nonample}). 

In any case, it is natural to ask what can be said about hyperbolicity for quasi-projective varieties. One can rephrase Conjecture \ref{conj:intro} for quasi-projective varieties $V$, and replace (3) with the condition that all subvarieties are of \emph{log general type}. We recall that a variety $V$ is of log general type if there exists a desingularization $\widetilde{V}$, and a projective embedding $\widetilde{V} \subset Y$ with $Y \setminus \widetilde{V}$ a divisor of normal crossings, such that $\omega_Y(D)$ is big.  It is then natural to ask if positivity of the \emph{log cotangent bundle} implies hyperbolicity in this setting. 

The first obstacle, is that the log cotangent sheaf is \emph{never ample}. However, one can essentially ask that this sheaf is ``as ample as possible'' (see Definition \ref{def:almostample} for the precise definition of almost ample). It turns out that, with this definition, quasi-projective varieties with almost ample log cotangent bundle are Brody hyperbolic (see \cite[Section 3]{Demailly1997}). In recent joint work with Kristin DeVleming \cite{advt}, we explore, among other things, the consequences for hyperbolicity that follow from such a positivity assumption. We prove that quasi-projective varieties with positive log cotangent bundle are arithmetically hyperbolic (see Theorem \ref{thm:main}), and that all their subvarieties are of log general type (see Theorem \ref{thm:almostample}).

\begin{theorem}\cite{advt} Let $(X,D)$ be a log smooth pair with almost ample $\Omega^1_X(\log D)$. If $Y \subset X$ is a closed subvariety, then 
\begin{enumerate} \item 
all pairs $(Y,E)$, where $E = (Y \cap D)_{red}$, with $Y \not\subset D$ are of log general type.
\item If in addition $\Omega^1_X(\log D)$ is globally generated, and $V \cong (X \setminus D)$ is a smooth quasi-projective variety over a number field $k$, then for any finite set of places $S$, the set of $S$-integral points $V(\calO_{k,S})$ is finite. \end{enumerate}
\end{theorem}

The main focus of these notes is to present, in a self-contained manner, the proofs of these statements. In particular, we review the notions of ampleness, almost ampleness, and global generation for vector bundles (see Section \ref{sec:hyplog}). The proof of the second statement heavily relies upon the theory of semi-abelian varieties and the quasi-Albanese variety, and so we develop the necessary machinery (see Section \ref{sec:semiabelian}).

Along the way, we discuss the related conjecture of Lang (see Conjecture \ref{conj:lang}), which predicts that varieties of general type \emph{do not} have a dense set of rational points. We discuss the (few) known cases in Section \ref{sec:knowncases}. A related, more general conjecture due to Vojta (see Conjecture \ref{conj:Vojta}) suggests that one can control the heights of points on varieties of general (resp. log general) type. We discuss this conjecture, and we introduce the theory of heights in Section \ref{sec:heights}. 

 We are then naturally led to discuss what happens in the \emph{function field} case (see Section \ref{sec:ff}). In this setting, the analogue of Faltings' Theorem is known (see Theorem \ref{thm:geo_mordell}). Similarly, positivity assumptions on the cotangent bundle lead to hyperbolicity results. In this context, we discuss a theorem of Noguchi (see Theorem \ref{th:noguchi}) and give insight into what is expected quasi-projective setting.

We first got interested in studying positivity of the log cotangent bundle to understand ``uniformity'' of integral points as it relates to the Lang-Vojta conjecture. Consequently, we end these notes with a short section discussing and summarizing some key results in this area. 

\subsection{Outline} The road map of these notes is the following:

\begin{itemize}
    \item[\S \ref{sec:rat_curves}] Rational points with a focus on projective curves.
    \item[\S \ref{sec:integral}] Integral points with a focus on quasi-projective curves. 
    \item[\S \ref{sec:canonical}] Tools from positivity and birational geometry.
    \item[\S \ref{sec:Lang}] Lang's conjecture and some known cases.
    \item[\S \ref{sec:hypproj}] Hyperbolicity of projective varieties and positivity of vector bundles. 
    \item[\S \ref{sec:hyplog}] Hyperbolicity for quasi-projective varieties.
    \item[\S \ref{sec:semiabelian}] Semi-abelian varieties, the quasi-Albanese, and arithmetic hyperbolicity of quasi-projective varieties. 
    \item[\S \ref{sec:vojta}] Vojta's conjecture and the theory of heights. 
    \item[\S \ref{sec:ff}] Diophantine geometry over function fields. 
    \item[\S \ref{sec:consequences}] Some known consequences of Lang's conjecture.
\end{itemize}

\subsection{Acknowledgments} These notes grew out of our minicourse given at the workshop ``Geometry and arithmetic of orbifolds'' which took place December 11-13, 2018 at the Universit\'e du Qu\'ebec \`a Montr\'eal.  We are grateful to the organizers: Steven Lu, Marc-Hubert Nicole, and Erwan Rousseau for the opportunity to give these lectures, and we thank the audience who were wonderful listeners and provided many useful comments which are reflected in these notes. We are grateful to Marc-Hubert Nicole for organizing this collection. We also thank Erwan Rousseau for pointing us to the paper \cite{cp}. We thank Damian Brotbek, Pietro Corvaja, Carlo Gasbarri, Ariyan Javanpeykar, and De-Qi Zhang for helpful discussions and comments; we are also grateful to the anonymous referees who greatly helped improving the first version of this notes. Finally, we thank our coauthor Kristin DeVleming whose work is represented here, and for many comments and conversations related to this work, our lectures, and the writing of these notes. Research of K.A. was supported in part by an NSF Postdoctoral Fellowship. Research of A.T. was supported in part by funds from NSF grant DMS-1553459 and in part by Centro di Ricerca Matematica Ennio de Giorgi. Part of these notes were written while K.A. was in residence at the Mathematical Sciences Research Institute in Berkeley, California, during the Spring 2019, supported by the National Science Foundation under Grant No. 1440140.

\subsection{Notation} We take this opportunity to set some notation.
\subsubsection{Geometry}
Divisors will refer to Cartier divisors, and $\Pic(X)$ will denote the Picard group, i.e. the group of isomorphism classes of line bundles on $X$. We recall that a reduced divisor is of \emph{normal crossings} if each point \'etale locally looks like the intersection of coordinate hyperplanes. 

\subsubsection{Arithmetic}
Throughout $k$ will denote a number field, i.e. a finite extension of $\Q$. We will denote by $M_k$ the set of places of $k$, i.e. the set of equivalence classes of absolute values of $k$. 
We will denote by $\calO_k = \{ \alpha \in k : \lvert \alpha \rvert_v \leq 1 \text { for every } v \in M_k \}$ the ring of integers of $k$, with $\calO_k^*$ the group of units, and with $\calO_{k,S}$ the ring of $S$-integers in $k$, i.e. $\calO_{k,S} = \{ \alpha \in k : \lvert \alpha \rvert_v \leq 1 \text{ for every } v \notin S \}$. We will denote an algebraic closure of $k$ by $\kbar$.

\section{Rational points on projective curves}\label{sec:rat_curves}
Some of the main objects of study in Diophantine Geometry are \emph{integral} and \emph{rational points} on varieties. An \emph{algebraic variety} is the set of common solutions to a system of polynomial equations with coefficients in $R$, where $R$ is usually a field or a ring. In these notes we will consider fields that are finite extensions of $\Q$, and rings that are finite extensions of $\Z$.

If $X$ is a variety defined over a number field $k$, i.e. defined by equations with coefficients in $k$, then the set of $k$-\emph{rational} points of $X$ is the set of solutions with coordinates in $k$. In a similar way, one can consider the set of \emph{integral} points of $X$ as the set of solutions with coordinates that belong to the ring of integers of $k$. However it is sometimes more convenient to take an approach that does not depend on the particular choice of coordinates used to present $X$. 

\begin{definition}[Rational points]\label{def:rat_points} Let $X$ be a projective variety defined over $k$. If $P \in X(\kbar)$ is an algebraic point, then the residue field $k(P)$ is a finite extension of $k$. We say that $P$ is \emph{$k$-rational} if $k(P) = k$. \end{definition}

\begin{remark} The above notion is intrinsic in the sense that it depends only on the function field of $X$, which is independent of the embedding in projective space. In this case a rational point corresponds to a morphism $\Spec k \to X$ (Exercise).\end{remark}

For a non-singular curve $\calC$ defined over a number field $k$, the genus governs the distribution of the $k$-rational points: if a curve is rational, i.e. it has genus zero, then the set of $k-$rational points $\calC(k)$ is dense, at most after a quadratic extension of $k$. Similarly, if the curve has genus one, then at most after a finite extension of $k$, the set of $k$-rational points $\calC(k)$ is dense (see \cite{Pra} for a gentle introduction and proofs of these statements). In the genus one case one can prove a stronger statement, originally proven by Mordell, and extended by Weil to arbitrary abelian varieties, namely that the set of $k$-rational points forms a finitely generated abelian group. We can summarize this in the following proposition.

\begin{prop}\label{prop:lowgenus} Let $\calC$ be a non-singular projective curve of genus $g(\calC)$ defined over a number field $k$. 
\begin{itemize}
\item If $g(\calC) = 0$, then $\calC(k)$ is dense, after at most a quadratic extension of $k$. 
\item If $g(\calC) = 1$, then $\calC(k)$ is a finitely generated group of positive rank, (possibly) after a finite extension of $k$. 
\end{itemize}
\end{prop}

In 1922, Mordell conjectured that a projective curve $\calC$ of genus $g(\calC) > 1$ has finitely many $k$-points. This was proven by Faltings' \cite{Falt}.

\begin{theorem}[Faltings' Theorem \cite{Falt}, formerly Mordell's conjecture]\label{thm:faltings}Let $\calC$ be a non-singular projective curve $\calC$ defined over a number field $k$. If $g(\calC) > 1$ then the set $\calC(k)$ is finite. \end{theorem}

The original proof of Faltings reduced the problem to the Shafarevich conjecture for abelian varieties, via Parhsin's trick. The argument uses very refined and difficult tools like Arakelov Theory on moduli spaces, semistable
abelian schemes and p-divisible groups, and therefore such a proof is outside the scope of these notes. A different proof was given shortly after by Vojta in \cite{Vojta_Mordell} using ideas from Diophantine approximation while still relying on Arakelov theory. Faltings in \cite{Faltings_ab} gave another simplification, eliminating the use of the arithmetic Riemann-Roch Theorem for arithmetic threefolds in Vojta's proof, and was able to extend these methods to prove a conjecture of Lang. Another simplification of both Vojta and Faltings' proofs was given by Bombieri in \cite{Bombieri} combining ideas from Mumford \cite{Mumford} together with the ones in the aforementioned papers.

The above results leave open many other Diophantine questions: when is the set $\calC(k)$ empty? Is there an algorithm that produces a set of generators for $E(k)$, for an elliptic curve $E$ defined over $k$? Is there an algorithm that computes the set $\calC(k)$ when it is finite (Effective Mordell)? We will not address these questions in this notes, but we will mention the very effective Chabauty-Coleman-Kim method that in certain situations can give answer to the latter question (see \cite{Chabauty, Coleman, McC-Poo, Kim, Kim2}).

%
%
%
%

\subsection{Geometry influences arithmetic}
In order to generalize, at least conjecturally, the distribution of $k$-rational points on curves to higher dimensional varieties, it is convenient to analyze the interplay between the arithmetic and the geometric properties of curves, following the modern philosophy that the geometric invariants of an algebraic variety determine arithmetic properties of the solution set.

We start by recalling the definition of the canonical sheaf.

\begin{definition} Let $X$ be a non-singular variety over $k$ of $\dim X = n$. We define the \emph{canonical sheaf} of $X$ to be $\omega_X = \bigwedge^n \Omega_{X/k}^1$, where $\Omega_{X/k}^1$ denotes the sheaf of relative differentials of $X$. \end{definition}

If $\calC$ is a curve, then $\omega_X = \Omega_{\calC}^1$ is an invertible sheaf whose sections are the global 1-forms on $\calC$. In this case, we call any divisor in the linear equivalence class a \emph{canonical divisor}, and denote the divisor by $K_\calC$. 

\begin{example} Let $\calC \cong \PP^1$, with coordinates $[x:y]$. In the open affine $U_x$ given by $x \neq 0$ we can consider the global coordinate $t = y/x$ and the global differential form $dt$. We can extend $dt$ as a \emph{rational} differential form $s \in \Omega^1_{\PP^1}$, noting that it will possibly have poles. To compute its associated divisor we note that in the locus $U_x \cap U_y$, i.e. where $x \neq 0$ and $y \neq 0$, the section $s$ is invertible. In the intersection, the basic formula $d\left( 1/t\right) = -dt/t^2$, shows that the divisor associated to $s$ is $-2P$, where $P = [0:1]$. In particular, given that any two points are linearly equivalent on $\PP^1$, $K_{\PP^1} \sim -P_1 -P_2$ for two points on $\PP^1$, and $\deg K_{\PP^1}= -2$.
\end{example}

  More generally, any divisor $D$ on a smooth curve $\calC$ is a weighted sum of points, and its degree is the sum of the coefficients. In the case of the canonical divisor, if $\calC$ is a curve of genus $g$, then $K_\calC$ has $\deg(K_\calC) = 2g-2$ (see \cite[Example IV.1.3.3]{hartshorne}). 

  Given Theorem \ref{thm:faltings} and Proposition \ref{prop:lowgenus}, one can see that the positivity of the canonical divisor $K_\calC$, determines the distribution of $k-$rational points $\calC(k)$. In particular, the set $\calC(k)$ is finite if and only if $\deg(K_\calC) > 0$. 

  There is a further geometric property of curves that mimics the characterization of rational points given above. If $\calC$ is a non-singular curve defined over a number field $k$, holomorphic maps to the corresponding Riemann Surface $\calC_\C$ (viewed either as the set of complex points $\calC(\C)$ together with its natural complex structure, or as the analytification of the algebraic variety $\calC$) are governed by the genus $g(\calC)$. If the genus is zero or one, the universal cover of $\calC_C$ is either the Riemann sphere or a torus, and therefore there exist non-constant holomorphic maps $\C \to \calC$ with dense image. On the other hand, if $g(\calC) \geq 2$, the universal cover of $\calC_\C$ is the unit disc and, by Liouville's Theorem, every holomorphic map to $\calC_\C$ has to be constant, since its lift to the universal cover is constant.

  Varieties $X$ where every holomorphic map $\C \to X$ is constant play a fundamental role in complex analysis/geometry so we recall here their definition.

  \begin{definition}\label{def:cx_hyp}
	  Let $X$ be a complex analytic space. We say that $X$ is \emph{Brody Hyperbolic} if every holomorphic map $X \to \C$ is constant. We say that $X$ is \emph{Kobayashi hyperbolic} if the Kobayashi pseudo-distance is a distance (see \cite{kobayashi} for definition and properties).   \end{definition}
	  
\begin{remark} When $X$ is compact the two notions are equivalent by \cite{brody} and we will only say that $X$ is hyperbolic. For more details about the various notions of hyperbolicity and their connection with arithmetic and geometric properties of varieties we refer to the chapter by Javanpeykar in this volume \cite{Javan_book} \end{remark}

  Given a non-singular projective curve $\calC$ defined over a number field, the previous discussion can be summarized in the following table.

\begin{table}[!htbp]
\begin{tabular}{|l|l|l|l|}
\hline
$g(\calC)$   & $\deg(K_\calC)$ & complex hyperbolicity & density of $k$-points \\ \hline
0        & $-2 < 0$    & not hyperbolic  & potentially dense     \\ \hline
1        & 0           & not hyperbolic & potentially dense     \\ \hline
$\geq 2$ & $2g-2 > 0$  & hyperbolic & finite                \\ \hline
\end{tabular}
\end{table}

\begin{remark} In the table, $\deg(K_\calC) >0$ is equivalent to requiring that $\omega_C$ is ample, since for curves a divisor is ample if and only if its degree is positive (see Section \ref{sec:posproper}).\end{remark}

\section{Integral points on curves}\label{sec:integral}
The previous section deals with the problem of describing the set $\calC(k)$, which we can think of the $k$-solutions of the polynomial equations that define $\calC$. An analogous problem, fundamental in Diophantine Geometry, is the study of the integral solutions of these equations, or equivalently of the integral points of $\calC$. However, the definition of integral point is more subtle.

\begin{example}\label{ex:P1_integral}
  Consider $\PP^1_\C$ as the set $\{ [x_0 : x_1] : x_0,x_1 \in \C \}$; if $k\subset \C$ is a number field, then the set of rational points, as defined in Definition \ref{def:rat_points}, is the subset $\PP^1_\C(k) \subset \PP^1_\C$ consisting of the points $[x_0:x_1]$ such that both coordinates are in $k$.
  
  \begin{remark}As points in $\PP^n$ are equivalence classes, we are implicitly assuming the choice of a representative with $k$-rational coordinates. For example the point $[\sqrt{2}:\sqrt{2}]$ is a $k$-rational point, being nothing but the point $[1:1]$.\end{remark}
  
  Let us focus on the case $k = \Q$. We want to identify the \emph{integral} points: it is natural to consider points $[x_0:x_1]$ in which both coordinates are \emph{integral}, i.e. $x_0,x_1 \in \calO_k = \Z$. In this case, by definition of projective space, this is equivalent to considering points of the form $[a:b]$ in which $\gcd(a,b) = 1$.

  Now consider the problem of characterizing integral points among $\Q$-rational points, i.e. given a point $P = [\frac{a}{b}: \frac{c}{d}] \in \PP^1(\Q)$, when is this point integral (assuming we already took care of common factors)? One answer is to require that $b=d=1$; however, the point $P$ can also be written as $P = [ad: bc]$, and since we assumed that we already cleared any common factor in the fractions, we have $\gcd(ad,bc) = 1$. So in particular, \emph{every} rational point is integral!
\end{example}

\begin{example}\label{ex:A1_integral}
  Let us consider the \emph{affine} curve $\Aff^1_\C \subset \PP^1_\C$, as the set $\{[a: 1]: a \in \C\}$. Then $\Aff^1(k) = \{ [a:1] : a \in k \}$. The integral points should correspond to $\{ [a:1] : a \in \calO_k \}$. Now we can ask the same question as in Example \ref{ex:P1_integral}, specializing again to $k = \Q$: namely how can we characterize integral points on $\Aff^1$ among its $\Q$-rational points? Given a rational point $P = [\frac{a}{b}:1]$, we can require that $b = 1$. This is equivalent to asking that for every prime $\fp \in \Z$, the prime $\fp$ does not divide $b$. In the case in which $\fp \mid b$, we can rewrite $P = [a: b]$ and we can see that the reduction modulo $\fp$ of $P$, i.e. the point whose coordinates are the reduction modulo $\fp$ of the coordinates of $P$, is the point $[1:0]$ which does not belong to $\Aff^1$! This shows that one characterization of integral points is the set of $k$-rational points whose reduction modulo every prime $\fp$ is still a point of $\Aff^1$. More precisely, let $D = [0:1] \in \PP^1 \setminus \Aff^1$ be the point at infinity: integral points $\Aff^1(\calO_k)$ are precisely the $k$-rational points $\Aff^1(k)$ such that their reduction modulo every prime of $\calO_k$ is disjoint from $D$.
\end{example}

The previous example gives an intuition for a coordinate-free definition of integral points. Note that, given an affine variety $X \subset \Aff^n$ defined over the ring of integers $\calO_k$ of a number field $k$ one could try to mimic Definition \ref{def:rat_points} as follows: define the ring $\calO_k[X]$ to be the image of the ring $\calO_k[T_1,\dots,T_n]$ inside the coordinate ring $k[X]$ of $X$. Now given a rational point $P = (p_1,\dots,p_n) \in X(k)$, it is integral when all the coordinates are in $\calO_k$. Therefore the point $P$ defines a specialization morphism $\varphi_p: \calO_k[T_1,\dots,T_n] \to \calO_k$ which induces, by passing to the quotient, a morphism $\varphi_p: \calO_k[X] \to \calO_k$. The construction can be reversed to show that indeed every such morphism corresponds to an integral point. Note that this definition depends on the embedding $X \subset \Aff^n$.

We will instead pursue a generalization that is based on Example \ref{ex:A1_integral}. Recall that the characterization we obtained of $\Aff^1(\calO_k)$ made use of the reduction modulo primes of points. To give a more intrinsic and formal definition we introduce the notion of \emph{models}.

\begin{definition}[Models]
  \label{def:models}
  Let $X$ be a quasi-projective variety defined over a number field $k$. A \emph{model} of $X$ over the ring of integers $\calO_k$ is a variety $\calX$ with a dominant, flat, finite type map $\calX \to \Spec \calO_k$ such that the generic fiber $\calX_\eta$ is isomorphic to $X$. See Figure \ref{fig:abramovich} for an illustration taken from \cite{dan}. 
\end{definition}

\begin{center}
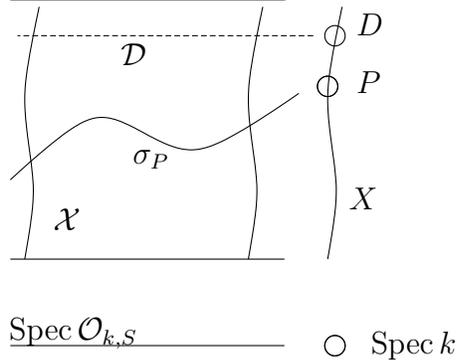
\begin{figure} \caption{An illustration of a model (from \cite{dan}).}\label{fig:abramovich}
\setlength{\unitlength}{0.0005in}
\begingroup\makeatletter\ifx\SetFigFont\undefined%
\gdef\SetFigFont#1#2#3#4#5{%
  \reset@font\fontsize{#1}{#2pt}%
  \fontfamily{#3}\fontseries{#4}\fontshape{#5}%
  \selectfont}%
\fi\endgroup%
{\renewcommand{\dashlinestretch}{30}
\begin{picture}(4285,3906)(0,-10)
\path(12,3879)(2862,3879)
\put(3312,2979){\ellipse{212}{212}}
\put(3312,2979){\ellipse{212}{212}}
\put(3387,279){\ellipse{212}{212}}
\put(3387,279){\ellipse{212}{212}}
\drawline(2862,3879)(2862,3879)
\path(12,1179)(2862,1179)
\path(12,279)(2862,279)
\path(3462,3804)(3462,3802)(3460,3797)
	(3458,3787)(3455,3773)(3451,3753)
	(3446,3727)(3439,3695)(3431,3658)
	(3423,3617)(3414,3573)(3405,3526)
	(3395,3478)(3386,3429)(3377,3381)
	(3368,3334)(3360,3288)(3353,3244)
	(3346,3202)(3339,3162)(3334,3124)
	(3329,3088)(3325,3054)(3321,3021)
	(3318,2989)(3316,2958)(3314,2928)
	(3313,2899)(3312,2870)(3312,2841)
	(3312,2811)(3313,2780)(3314,2749)
	(3316,2718)(3318,2686)(3320,2654)
	(3323,2621)(3327,2588)(3330,2555)
	(3334,2521)(3338,2487)(3342,2453)
	(3347,2418)(3351,2384)(3356,2350)
	(3360,2316)(3365,2283)(3369,2250)
	(3374,2217)(3378,2185)(3382,2154)
	(3385,2124)(3388,2094)(3391,2065)
	(3394,2036)(3396,2009)(3398,1981)
	(3400,1954)(3401,1925)(3401,1896)
	(3401,1867)(3400,1837)(3399,1807)
	(3397,1776)(3395,1743)(3392,1709)
	(3388,1673)(3384,1635)(3378,1595)
	(3373,1553)(3367,1510)(3360,1466)
	(3353,1422)(3346,1379)(3339,1338)
	(3333,1300)(3327,1266)(3322,1237)
	(3318,1215)(3315,1198)(3314,1188)
	(3312,1182)(3312,1179)
\path(2637,3804)(2637,3802)(2635,3797)
	(2633,3787)(2630,3773)(2626,3753)
	(2621,3727)(2614,3695)(2606,3658)
	(2598,3617)(2589,3573)(2580,3526)
	(2570,3478)(2561,3429)(2552,3381)
	(2543,3334)(2535,3288)(2528,3244)
	(2521,3202)(2514,3162)(2509,3124)
	(2504,3088)(2500,3054)(2496,3021)
	(2493,2989)(2491,2958)(2489,2928)
	(2488,2899)(2487,2870)(2487,2841)
	(2487,2811)(2488,2780)(2489,2749)
	(2491,2718)(2493,2686)(2495,2654)
	(2498,2621)(2502,2588)(2505,2555)
	(2509,2521)(2513,2487)(2517,2453)
	(2522,2418)(2526,2384)(2531,2350)
	(2535,2316)(2540,2283)(2544,2250)
	(2549,2217)(2553,2185)(2557,2154)
	(2560,2124)(2563,2094)(2566,2065)
	(2569,2036)(2571,2009)(2573,1981)
	(2575,1954)(2576,1925)(2576,1896)
	(2576,1867)(2575,1837)(2574,1807)
	(2572,1776)(2570,1743)(2567,1709)
	(2563,1673)(2559,1635)(2553,1595)
	(2548,1553)(2542,1510)(2535,1466)
	(2528,1422)(2521,1379)(2514,1338)
	(2508,1300)(2502,1266)(2497,1237)
	(2493,1215)(2490,1198)(2489,1188)
	(2487,1182)(2487,1179)
\path(312,3804)(312,3802)(310,3797)
	(308,3787)(305,3773)(301,3753)
	(296,3727)(289,3695)(281,3658)
	(273,3617)(264,3573)(255,3526)
	(245,3478)(236,3429)(227,3381)
	(218,3334)(210,3288)(203,3244)
	(196,3202)(189,3162)(184,3124)
	(179,3088)(175,3054)(171,3021)
	(168,2989)(166,2958)(164,2928)
	(163,2899)(162,2870)(162,2841)
	(162,2811)(163,2780)(164,2749)
	(166,2718)(168,2686)(170,2654)
	(173,2621)(177,2588)(180,2555)
	(184,2521)(188,2487)(192,2453)
	(197,2418)(201,2384)(206,2350)
	(210,2316)(215,2283)(219,2250)
	(224,2217)(228,2185)(232,2154)
	(235,2124)(238,2094)(241,2065)
	(244,2036)(246,2009)(248,1981)
	(250,1954)(251,1925)(251,1896)
	(251,1867)(250,1837)(249,1807)
	(247,1776)(245,1743)(242,1709)
	(238,1673)(234,1635)(228,1595)
	(223,1553)(217,1510)(210,1466)
	(203,1422)(196,1379)(189,1338)
	(183,1300)(177,1266)(172,1237)
	(168,1215)(165,1198)(164,1188)
	(162,1182)(162,1179)
\path(12,2004)(14,2006)(18,2010)
	(26,2017)(38,2028)(55,2044)
	(77,2064)(104,2089)(136,2117)
	(172,2150)(211,2185)(254,2222)
	(298,2260)(343,2298)(388,2336)
	(433,2373)(477,2408)(519,2441)
	(560,2473)(599,2501)(637,2527)
	(673,2551)(707,2571)(739,2590)
	(770,2606)(800,2619)(829,2630)
	(857,2639)(884,2646)(910,2650)
	(936,2653)(962,2654)(989,2653)
	(1017,2650)(1045,2646)(1072,2639)
	(1101,2632)(1129,2622)(1158,2612)
	(1187,2599)(1217,2586)(1248,2572)
	(1278,2557)(1310,2541)(1341,2524)
	(1373,2507)(1405,2490)(1437,2472)
	(1469,2455)(1501,2439)(1533,2422)
	(1565,2407)(1597,2392)(1628,2378)
	(1660,2365)(1691,2354)(1721,2344)
	(1752,2335)(1783,2328)(1813,2322)
	(1844,2319)(1875,2316)(1902,2316)
	(1930,2318)(1958,2320)(1987,2325)
	(2017,2332)(2048,2340)(2081,2350)
	(2114,2362)(2150,2377)(2187,2393)
	(2226,2411)(2266,2432)(2309,2455)
	(2354,2480)(2401,2507)(2450,2535)
	(2500,2566)(2551,2597)(2603,2630)
	(2654,2663)(2705,2696)(2754,2728)
	(2800,2759)(2843,2788)(2882,2814)
	(2916,2837)(2944,2857)(2968,2873)
	(2985,2885)(2998,2894)(3006,2900)
	(3010,2903)(3012,2904)
\dashline{60.000}(3162,3504)(87,3504)
\put(3387,3504){\ellipse{212}{212}}
\put(3762,129){\makebox(0,0)[lb]{{\SetFigFont{12}{14.4}{\rmdefault}{\mddefault}{\updefault}$\Spec k$}}}
\put(0,254){\makebox(0,0)[lb]{{\SetFigFont{12}{14.4}{\rmdefault}{\mddefault}{\updefault}$\Spec \calO_{k,S}$}}}
\put(3612,3504){\makebox(0,0)[lb]{{\SetFigFont{12}{14.4}{\rmdefault}{\mddefault}{\updefault}$D$}}}
\put(1162,3204){\makebox(0,0)[lb]{{\SetFigFont{12}{14.4}{\rmdefault}{\mddefault}{\updefault}$\calD$}}}
\put(3612,2904){\makebox(0,0)[lb]{{\SetFigFont{12}{14.4}{\rmdefault}{\mddefault}{\updefault}$P$}}}
\put(1287,2129){\makebox(0,0)[lb]{{\SetFigFont{12}{14.4}{\rmdefault}{\mddefault}{\updefault}$\sigma_P$}}}
\put(3537,1704){\makebox(0,0)[lb]{{\SetFigFont{12}{14.4}{\rmdefault}{\mddefault}{\updefault}$X$}}}
\put(462,1479){\makebox(0,0)[lb]{{\SetFigFont{12}{14.4}{\rmdefault}{\mddefault}{\updefault}$\calX$}}}
\end{picture}
}
\end{figure}
\end{center}

Over every prime $\fp$ of $\calO_k$ we have a variety $\calX \times_{\calO_k} \Spec k_\fp$ defined over the residue field $k_\fp$ which is the ``reduction modulo $\fp$ of $X$'', while the generic fiber over $(0)$ is isomorphic to the original $X$. This will make precise the notion of reduction modulo $\fp$ of a prime.

Given a rational point $P \in X(\kappa)$, since $X \cong \calX_\eta$, this gives a point in the generic fiber of a model $\calX$ of $X$. If the model $\calX$ is proper, the point $P$, which corresponds to a map $\Spec k \to X$, will extend to a section $\sigma_P: \Spec \calO_k \to \calX$, therefore defining the reduction modulo a prime of $P$: it is just the point $P_\fp = \sigma(\calO_k) \cap \calX_\fp$. Therefore, given a \emph{proper} model of $X$, there is a well defined notion of the reduction modulo a prime of $k$-rational points.

The last ingredient we need to define integral points is the analogue of the point $D = [0:1] \in \PP^1 \setminus \Aff^1$ in Example \ref{ex:A1_integral}. In the example we used that the affine curve $\Aff^1$ came with a (natural) compactification, namely $\PP^1$, and a divisor ``at infinty'', namely $D$. This motivates the idea that to study integral points, the geometric objects that we need to consider are pairs of a variety and a divisor, i.e. objects of the form $(X,D)$ where $X$ is a projective variety (corresponding to the compactification of the affine variety) and $D$ is a divisor (corresponding to the divisor ``at infinity''). 

\begin{definition}\label{def:pair}
  A \emph{pair} is a couple $(X,D)$ where $X$ is a (geometrically integral) projective variety defined over a field $k$ and a normal crossing divisor $D$. A model $(\calX,\calD) \to \Spec \calO_k$ of the pair, is a proper model $\calX \to \Spec \calO_k$ of $X$ together with a model $\calD \to \Spec \calO_k$ of the divisor $D$ such that $\calD$ is a Cartier divisor of $\calX$. 
\end{definition}

\begin{remark}\label{rmk:qproj} Given a non-singular affine variety $Y$ defined over a field of characteristic 0, combining the theorems of Nagata and Hironaka, we can always find a non-singular projective compactification $X \supset Y$ such that $D = X \setminus Y$ is a simple normal crossing divisor. Therefore we can identify (non canonically) every non-singular affine variety $Y$ as the pair $(X,D)$. This gives a way to characterize the set of integral points on $Y$ as the set of rational points of $X$ whose reduction modulo every prime does not specialize to (the reduction of) $D$. Using models, this gives a formal intrinsic defining of integral points. \end{remark}

\begin{definition}\label{def:int_points} 
	Given a pair $(X,D)$, the \emph{set of $D$-integral points} of $X$, or equivalently the \emph{set of integral points} of $Y = X \setminus D$, with respect to a model $(\calX,\calD)$ of $(X,D)$, is the set of sections $\Spec \calO_k \to \calX \setminus \calD$. We will denote this set by $X(\calO_{k,D})$ or $Y(\calO_k)$ for $Y = X \setminus D$.
\end{definition}

\begin{remark}\leavevmode
  \begin{itemize}
    \item In the case in which $Y$ is already projective, i.e. $X = Y$ and $D = \emptyset$, then the set of integral points coincide with the set of \emph{all} sections $\Spec \calO_k \to \calX$. Since the model $\calX \to \Spec \calO_k$ is proper, this is equivalent to the set of sections of the generic fiber, i.e. maps $\Spec k \to X$, which is exactly the set $X(k)$. This shows that for projective varieties, the set of rational points coincide with the set of integral points.
    \item The definition of integral points depends on the choice of the model! Different choices of models might give different sets of integral points (see Example \ref{ex:blowup}).
    \item When we consider an affine variety $V$ given inside an affine space $\Aff^n$ as the vanishing of a polynomial $f$ with coefficients in $\calO_k$, this corresponds to a pair $(X,D)$ where $X$ is the projective closure of $V$ (i.e. the set of solutions to the equations obtained by homogenizing $f$), and $D$ the boundary divisor. In this case there is a natural model of $(X,D)$ given by the model induced by the natural model of $\PP^n$ over $\Spec \calO_k$, i.e. $\PP^n_{\calO_k}$ and the closure of $D$ inside it. Then one can show (exercise) that the set of integral points with respect to this model coincide with solutions of $f = 0$ with integral coordinates (see discussion before Definition \ref{def:models}).
  \end{itemize}
  
  The following example shows that the definition of integral points as above truly depends on the choice of model.
  
	    \begin{example}[Abramovich]\label{ex:blowup}
		    Consider an elliptic curve given as $E : y^2 = x^3 + A x + B$ with $A,B \in \Z$, and as usual the origin will be the point at infinity. Since $E$ is given as the vanishing of a polynomial equation, the homogenization defines a closed subset of $\PP^2_\Q$. Moreover, since the coefficients are all integers we get that the same equation defines a model $\calE$ of $E$ which is the closure of $E$ inside $\PP^2_\Z$, i.e. the standard model of $\PP^2_\Q$ over $\Spec \Z$. Let $P \in E(\Q)$ be a rational point which is not integral with respect to $D = \{ 0_E \}$. This means that there exists a prime $\fp \in \Spec \Z$ such that $P$ reduces to the origin modulo $\fp$. In particular, the section $P: \Spec \Z \to \calE$ intersects the zero section over the prime $\fp$. Call $Q$ the point of intersection.

		    Consider now the blow up $\pi: \calE' \to \calE$ of $\calE$ at $Q$: by definition of the blow up $\calE'$ is also a model of $E$. To see this, observe that the composition of $\pi$ with the model map $\calE \to \Spec \Z$ is still flat and finite type; moreover the point we blow up was in a special fiber so it did not change the generic fiber, which is still isomorphic to $E$. In this new model the lift of the section $P: \Spec \Z \to \calE'$ does not intersect the zero section over the prime $\fp$. We can repeat this process for every point where the section $P(\Spec \Z)$ intersects the zero section, which will result in a different model for which the point $P$ is now an integral point! This shows that the notion of integral points depends on the model chosen.
  \end{example}
\end{remark}

%
%

  The following example motivates studying $(S,D)$-integral points, where $S$ is a finite set of places.

  \begin{example}Sometimes it is useful to consider rational points that fail to be integral only for a specific set of primes in $\calO_k$. For example the equation $2x + 2y = 1$ does not have any integral solutions while having infinitely many rational solutions. However, one sees that it has infinitely many solutions in the ring $\Z[\frac{1}{2}]$, which is finitely generated over $\Z$. A solutions in $\Z[\frac{1}{2}]$, e.g. $(\frac{1}{4},\frac{1}{4})$, fails to be integral only with respect to the prime $2$. More precisely consider the model of $\calC: 2x + 2y - z = 0$ in $\PP^2_\Z$ and of the divisor $\calD = [1:-1:0]$: then $P = [1:1:4]$ is a $\Q$-rational point of $\calC$ but it is not integral, since the reduction of $P$ modulo $2$ is the point $P_2 = [1:1:0] = [1:-1:0]$ over $\F_2$. On the other hand, for every prime $\fp \neq 2$, the reduction modulo $\fp$ of $P$ is disjoint from $\calD$. 
  
  Analogously, the rational point $P$ gives rise, since the model is proper, to a morphism $P: \Spec\Z \to \calC$ which is not disjoint from $\calD$, but such that the intersection $P(\Spec\Z) \cap \calD$ is supported over the prime $2$. \end{example}

  This motivates the following definition:

  \begin{definition}
    Let $S$ be a finite set of places of $k$, and let $(\calX,\calD)$ be a model of a pair $(X,D)$ defined over $k$. An $(S,D)$-integral point is an integral point $P: \Spec \calO_k \to \calX$ such that the support of $P^* \calD$ is contained in $S$. We denote the set of $(S,D)$-integral points of $(X,D)$ as $X(\calO_{S,D})$. If $Y = X \setminus D$, then we denote the set of $(S,D)$-integral points of $Y$ as $Y(\calO_{S})$.
  \end{definition}
  
  \begin{remark}
    One can also define the set of $(S,D)$-integral points as sections $\Spec\calO_{k,S} \to \calX$ that do not intersect $\calD$, where $\calO_{k,S}$ is the ring of $S$-integers. 
  \end{remark}

  Now that we have an intrinsic definition for integral points we can concentrate on the problem of describing the set of $(S,D)$-integral points on affine curves. As in the case of rational points on projective curves, the distribution of integral points will be governed by the geometry of the affine curve, i.e. of the corresponding pair. For one dimensional pairs, the fundamental invariant is the Euler characteristic, or equivalently the degree of the log canonical divisor.
  
  \begin{definition}
    \label{def:Euler_char}
    Given a non-singular projective curve $\calC$ and a pair $(\calC,D)$, the Euler Characteristic of $(\calC,D)$ is the integer $\chi_D(\calC) = 2g(\calC) -2 + \#D$, which corresponds to the degree of the \emph{log canonical divisor} $K_\calC + D$. 
  \end{definition}

  The Euler Characteristic encodes information of both the genus of the projective curve $\calC$ and of the divisor $D$, and its sign determines the arithmetic of the affine curve $\calC \setminus D$.

  \begin{theorem}
    Given a pointed projective curve $(\calC,D)$ defined over a number field $k$ and a finite set of places $S$ the following hold:
    \label{th:int_curves}
\begin{itemize}
  \item If $2g(\calC) - 2 + \# D \leq 0$ then the set of $(S,D)$-integral points is dense, possibly after a finite extension of $k$ and/or $S$;
  \item If $2g(\calC) - 2 + \# D > 0$ then the set of $(S,D)$-integral points is finite (Siegel's Theorem).
\end{itemize}
  \end{theorem}

  We treat the case of non positive Euler Characteristic in the following example.

\begin{example}
  When $\calC$ is smooth projective, in order for $\chi_D(\calC)$ to be non positive, there are only four cases to consider: if $g(\calC) = 0$, then $\# D \leq 2$, and if the $g(\calC)=1$ then $D = \emptyset$. For projective curves, i.e. when $D = \emptyset$, Proposition \ref{prop:lowgenus} shows that, up to a finite extension of $k$, the rational points are infinite. We showed that for projective varieties, integral and rational points coincide, which implies that in these cases the set of integral points is dense, up to a finite extension of the base field.
  
  If we consider affine curves, i.e. such that $D \neq \emptyset$, then there are only two remaining cases that we have to discuss, namely $\Aff^1 = (\PP^1, P)$ and $\G_m = (\PP^1, P + Q)$.

  We saw in Example \ref{ex:A1_integral} that integral points on $\Aff^1$ are infinite, and more generally we have that $\Aff^1(\calO_{S,D}) \cong \calO_{k,S}$. In the case of the multiplicative group $\G_m$ the integral points correspond to the group of $S$-units $\calO_{k,S}^*$. To see this consider $\G_m$ as the complement of the origin in $\Aff^1$, i.e. $\PP^1 \setminus \{ [0:1],[1:0] \}$. Then a point $[a:b] \in \G_m(k)$ is $(S,D)$-integral if, for every $\fp$ in $\calO_{k,S}$, we have that $\fp$ does not divide neither $a$ or $b$, i.e. $a$ and $b$ are both $S$-units. Finally, Dirichlet's Unit Theorem, implies that the group of $S$-units is finitely generated and has positive rank as soon as $\# S \geq 2$. In particular, for every number field $k$, there exists a finite extension for which the rank of $\calO_{k,S}^*$ is positive, and therefore such that $\G_m(\calO_{S,D})$ is infinite.
\end{example}

In the following example, we show that the set of $(S,D)$-integral points on the complement of three points in $\bP^1$ is finite.

\begin{example}[$\bP^1$ and three points]
  Consider the case of $(\PP^1, D)$ where $D = [0] + [1] + [\infty]$ over a number field $k$. In this case $\deg(K_{\PP^1} + D) > 0$, therefore Siegel's Theorem tells us that the number of $(S,D)$-integral points is finite, for every finite set of places $S$ of $k$. This can be deduced directly in this case using the $S$-unit equation Theorem as follows.
  
  Integral points in the complement of $[\infty]$ are integral points in $\mathbb{A}^1$ with respect to the divisor $[0] + [1]$, i.e. $u \in k$ such that $u \in \calO_S^*$ (which corresponds to  integrality with respect to $[0]$) and such that $1-u \in \calO_S^*$ (which corresponds to  integrality with respect to $[1]$). Then if we define $v = 1 - u \in \calO_S^*$, the set of $(S,D)$-integral points correspond to solutions in $S$-units of $x + y = 1$. The $S$-unit Theorem (see e.g. \cite[Theorem 1.2.4]{Corvaja_book}) then implies that the set of solutions is always finite, for every set of places $S$. 
\end{example}


Siegel's Theorem \cite{Siegel} (and for general number fields and set of places $S$ in \cite{Mahler, Lang_intpoints}) on the finiteness of the set of $(S,D)$-integral points is the analogue of Faltings' Theorem \ref{thm:faltings}. We give here a brief sketch of the proof. For the details we refer to \cite[D.9]{HindrySilverman}.

\begin{proof}[Sketch of proof of Siegel's Theorem]
  We will focus on the case $g(\calC) \geq 1$; the case of genus zero can be treated via finiteness of solutions of $S$-unit equations, see \cite[Theorem D.8.4]{HindrySilverman}.
 We can always assume that $\calC$ has at least one rational point, and use the point to embed $\calC \to \Jac (\calC)$, in its Jacobian.

 Suppose that $(x_i)$ is an infinite sequence of integral points on $\calC \setminus \calD$. Then by completeness of $\calC(k_v)$, with $v \in S$, up to passing to a subsequence, $(x_i)$ converges to a limit $\alpha \in \calC(\overline{k_v})$. In the embedding $\calC \subset \Jac(\calC)$, we see that for every positive integer $m$ the sequence $(x_i)$ becomes eventually constant in $\Jac(\calC)/m\Jac(\calC)$, which is finite by the Weak Mordell-Weil Theorem. In particular we can write $x_i = m y_i + z$, for some fixed $z \in \Jac(\calC)$.

 Let $\varphi_m:\Jac(\calC) \to \Jac(\calC)$ be the multiplication by $m$ map and define $\psi(x) = \varphi_m(x) + z = m.x + z$. Then $(y_i)$ are a sequence of integral points (since $\psi$ is unramified, and applying Chevalley-Weil Theorem, see \cite[Theorem 1.3.1]{Corvaja_book}) on $\psi^*(\calC)$ that converges to some $\beta \in \Jac(\calC)$ (eventually up to passing to an extension).

 By definition of the canonical height on $\Jac(\calC)$ (with respect to a fixed symmetric divisor) one has that $\hat{h}(\psi(y_i)) \gg m^2 \hat{h}(y_i)$. By increasing $m$ one gets very good approximations to $\alpha$ which eventually contradicts Roth's Theorem \cite[Theorem D.2.1]{HindrySilverman}.\end{proof}

The sketch of the proof illustrates a couple of very powerful ideas in Diophantine Geometry: the use of abelian varieties (here played by $\Jac(\calC)$ as ambient spaces with extra structure), the use of the so-called height machine, and techniques from Diophantine approximation. A different proof that avoids the use of the embedding in the Jacobian, thus allowing generalization to higher dimensions, has been given more recently by Corvaja-Zannier in \cite{CZSiegel}, replacing Roth's Theorem by the use of Schmidt's Subspace Theorem (see \cite[Chapter 3]{Zannier_book} for more details).

Finally, we can ask about hyperbolicity properties of affine curves, as in Definition \ref{def:cx_hyp}; it is easy to see that both $\Aff^1$ and $\G_m$ are not hyperbolic (via the exponential map), while on the other hand the complement of any number of points in a curve of genus one is hyperbolic (again applying Liouville's Theorem). Therefore, if $(\calC,D)$ is a pair of a non-singular projective curve $\calC$ and a reduced divisor $D$, both defined over a number field $k$, and $S$ is a finite set of places containing the archimedean ones, we can summarize the result described in the previous sections in the following table:

\begin{table}[h!]
\begin{tabular}{|l|l|l|}
\hline
 $\chi_D(\calC) = \deg(K_\calC + D)$ & complex hyperbolicity & $(S,D)$-integral points \\ \hline
 $ \leq 0$    & not hyperbolic  & potentially dense     \\ \hline
 $ > 0$    & hyperbolic  & finite \\ \hline
\end{tabular}
\end{table}

\section{Positivity of the canonical sheaf}\label{sec:canonical}
As we saw for curves, hyperbolicity was governed by the positivity of the canonical sheaf. In particular, we saw if $g(C) \leq 1$ then $\deg \omega_C \leq 0$ (and $C$ is \emph{not} hyperbolic), and if $g(C) \geq 2$ then $\deg \omega_C > 0$ (and $C$ \emph{is} hyperbolic). Conjecturally, postivity of the canonical sheaf governs hyperbolicity of algebraic varieties.  Before introducing the conjectures, we give a few examples of canonical sheaves on proper algebraic varieties. Recall we saw earlier that for a curve $C$, the canonical sheaf $\deg\omega_C = 2g -2$.

\begin{example}\cite[Example II.8.20.1]{hartshorne}\label{ex:canpn} First consider the Euler sequence 
\[ 0 \to \calO_{\bP^n} \to \calO_{\bP^n}(1)^{\oplus(n+1)} \to \calT_{\bP^n} \to 0, \] where $\calT_{\bP^n}$ denotes the tangent sheaf. Taking highest exterior powers, we see that $\omega_{\bP^n} = \calO_{\bP^n}(-n-1)$.\end{example}

\begin{example} If $A$ is an abelian variety, then the tangent bundle of $A$ is trivial. In particular, $\omega_X = \calO_X$. \end{example}

A standard way to calculate the canonical sheaf of algebraic varieties is via the \emph{adjunction formula}, which relates the canonical sheaf of a variety to the canonical sheaf of a hypersurface inside the variety. 

If $X$ is smooth and projective, and $Y$ is a smooth subvariety, then there is an inclusion map $i: Y \hookrightarrow X$. If we denote by $\calI$ the ideal sheaf of $Y \subset X$, then the conormal exact sequence gives (where $\Omega_X$ denotes the cotangent sheaf on $X$) 
\[ 0 \to \calI/\calI^2 \to i^*\Omega_X \to \Omega_Y \to 0.\] 
In particular, taking determinants yields 
\[ \omega_Y = i^*\omega_X \otimes \det(\calI/\calI^2)^{\vee}.\]

If we let the subvariety $Y$ to be a divisor $D \subset X$, then one obtains the following.

\begin{prop}[Adjunction formula] Let $X$ be a smooth projective variety with $D$ a smooth divisor on $X$. Then $$K_D = (K_X + D)|_D.$$ \end{prop}

\begin{example}\label{ex:ci} We can use the adjunction formula to calculate that the canonical sheaf of $X$ a smooth hypersurface of degree $d$ in $\bP^n$ is $\omega_X \cong \calO_X(d-n-1)$.
We note one can do similar calculations in the case of complete intersections. 
\end{example}

Now that we have shown some examples of computing canonical sheaves, we introduce the notions from birational geometry we will use to understand positivity of the canonical sheaf and hyperbolicity.  Our main reference is \cite{Lazarsfeld}. 

\subsection{Notions from birational geometry}
Let $X$ be a projective variety and let $L$ be a line bundle on $X$. For each $m \geq 0$ such that $h^0(X, L^{\otimes m}) \neq 0$, the linear system $| L^{\otimes m}|$ induces a rational map \[\phi_m = \phi_{|L^{\otimes m}|} : X \dashrightarrow \bP H^0(X, L^{\otimes m}).\] We denote by $Y_m = \phi_m(X, L)$ the closure of its image. 

\begin{definition} Let $X$ be normal. The \emph{Iitaka dimension} of $(X,L)$ is 
\[ \kappa(X, L) = \max_{m > 0} \{ \dim \phi_m(X,L) \}, \] as long as $\phi_m(X,L) \neq \emptyset$ for some $m$. Otherwise, we define $\kappa(X,L) = - \infty$. \end{definition}

In particular, either $\kappa(X,L) = -\infty$ or $0 \leq \kappa(X,L) \leq \dim X$. 

\begin{remark} If $X$ is not normal, consider the normalization $\nu: X^{\nu} \to X$ and take $\kappa(X^{\nu}, \nu^*L)$. \end{remark}

\begin{example}[Kodaira dimension] If $X$ is a smooth projective variety and $K_X$ is the canonical divisor, then $\kappa(X, K_X) = \kappa(X)$ is the \emph{Kodaira dimension} of $X$. \end{example}

The Kodaira dimension is a birational invariant, and the Kodaira dimension of a singular variety $X$ is defined to be $\kappa(X')$ where $X'$ is any desingularization of $X$. However, care needs to be taken in this case. When $X$ is not smooth, the dualizing sheaf $\omega_X$ can exist as a line bundle on $X$, but $\kappa(X, \omega_X) > \kappa(X)$. This is the case, e.g. if $X$ is the cone over a smooth plane curve of large degree (see \cite[Example 2.1.6]{Lazarsfeld}). 

\subsubsection{Positivity of line bundles}\label{sec:posproper}

\begin{definition}\label{def:lbundles} A line bundle $L$  on a projective variety $X$ is \emph{ample} if for any coherent sheaf $F$ on $X$, there exists an integer $n_F$ such that $F \otimes L^{\otimes n}$ is generated by global sections for $n > n_F$. Equivalently, $L$ is ample if a positive tensor power is very ample, i.e. there is an embedding $j: X \to \bP^N$ such that $L^{\otimes n} = j^*(\calO_{\bP^N}(1))$. \end{definition}

The following result is a standard way for checking if a divisor is ample. 

\begin{theorem}[Nakai-Moishezon]\label{thm:NM} Let $X$ be a projective scheme and let $D$ be a divisor. The divisor $D$ is ample if and only if $D^{\dim Y}.Y > 0$ for all subvarieties $Y \subset X$. \end{theorem}

\begin{corollary} If $X$ is a surface, then a divisor $D$ is ample if and only if $D^2 > 0$ and $D. C > 0$ for all curves $C \subset X$. \end{corollary}

\begin{example} By Riemann-Roch, a divisor $D$ on a curve $C$ is ample if and only if $\deg D >0$. \end{example}

\begin{example} We saw in Example \ref{ex:canpn} that $\omega_{\bP^n} = \calO_{\bP^n}(-n-1)$. Therefore, we see that $\omega_{\bP^n}$ is \emph{never ample} for any $n$, as no power of $\omega_{\bP^n}$ will have non-zero sections. It is not so hard to see that $-\omega_{\bP^n}$ is ample for all $n$. This is referred to as \emph{anti-ample}. \end{example} 

\begin{example} In Example \ref{ex:ci} we computed the canonical sheaf for hypersurfaces of degree $d$ in $\bP^n$ using the adjunction formula. From this we see that
\begin{enumerate}
\item If $d \leq n$ then $\omega_X$ is anti-ample. 
\item if $d = n + 1$ then $\omega_X = \calO_X$, and thus is not ample. 
\item If $d \geq n + 2$ then $\omega_X$ is very ample (exercise using Serre Vanishing).
\end{enumerate}
\end{example}

\begin{definition}\label{def:big} A line bundle $L$ on a projective variety $X$ is \emph{big} if $\kappa(X,L) = \dim X$. A Cartier divisor $D$ on $X$ is \emph{big} if $\calO_X(D)$ is big. \end{definition}

\begin{remark}\label{rmk:big}There are some standard alternative criteria for big divisors. One is that there exists a constant $C > 0$ such that $h^0(X, \calO_X(mD)) \geq C \cdot m^n$ or all sufficiently large $m$ (see \cite[Lemma 2.2.3]{Lazarsfeld}). Another is that $mD$ can be written as the sum of an ample plus effective divisor (Kodaira's Lemma, see \cite[Corollary 2.2.7]{Lazarsfeld}).\end{remark}

\begin{definition} We say that $X$ is of \emph{general type} if $\kappa(X) = \dim(X)$, i.e. $\omega_X$ is \emph{big}.  \end{definition}

\begin{example} We see immediately that ample implies big so that varieties with ample canonical sheaves are of general type. In this case, some power $\omega_X^{\otimes m}$ for $m \gg 0$ embeds $X$ into a projective space.  \end{example}

\begin{example} For curves big is the same as ample, so general type is equivalent to $g(C) \geq 2$. \end{example}

Some examples of varieties of general type are high degree hypersurfaces (in $\bP^3$ we require $d \geq 5$) and products of varieties of general type (e.g. product of higher genus curves). It is worth noting that projective space $\bP^n$ and abelian varieties are \emph{not} of general type.  
 
\begin{remark}\label{rmk:pullback}
There exist big divisors that are not ample. One of the standard ways to obtain examples is to note that bigness is preserved under pullback via birational morphisms, but ampleness is not. Suppose $X$ and $Y$ are proper, and $f: X \to Y$ is a birational morphism. A divisor $D$ on $Y$ is big if and only if $f^*D$ is big on $X$. This is easy to see using Definition \ref{def:big}, since $X$ and $Y$ have isomorphic dense open subsets. 

We now give an example to show that ampleness is not preserved. Suppose $H$ be a line in $\bP^2$, and let $f: X \to \bP^2$ be the blowup of $\bP^2$ at a point with exceptional divisor $E$. Then $f^*H$ is big by the above discussion, but $f^*H$ is \emph{not ample} since the projection formula gives that $(f^*H).E = 0$, and thus violates Theorem \ref{thm:NM}. 
\end{remark}

\subsection{Log general type}
As we saw for proper curves, hyperbolicity was essentially governed by the positivity of the canonical sheaf. For affine curves, we saw that hyperbolicity was governed by positivity of the log canonical divisor. As a result, we discuss a mild generalization that will be needed later -- the notion of \emph{log general type} for quasi-projective varieties. Recall that we saw in Remark \ref{rmk:qproj} that given a quasi-projective variety $V$, one can always relate it to a pair $(X,D)$ of a smooth projective variety and normal crossings divisor $D$. 

\begin{definition}\label{def:loggt}
We say that $V$ (or the pair $(X,D)$) is of \emph{log general type} if $\omega_X(D)$ is big. 
\end{definition}

Of course any pair $(X,D)$ with $X$ of general type will be of log general. Perhaps more interesting examples are when $X$ does not have its own positivity properties. 

\begin{example}[Curves]
A pointed curve $(C, D = \sum p_i)$ is of log general type if:
\begin{itemize}
\item $g(C) = 0$ and $\#(D) \geq 3$, 
\item $g(C) = 1$ and $\#(D) \geq 1$, or
\item $g(C) \geq 2$. 
\end{itemize}
This is because $\deg\omega_C = 2g-2$ and so $\deg \omega_C(D)  = 2g-2+\#D$. 
\end{example}

\begin{example} If $X = \bP^2$ and $D$ is a normal crossings divisor, then the  pair $(X, D)$ is of log general type if the curve $D$ has $\deg(D) \geq 4$. Again, this is because $\omega_X(D) \cong \calO_{\bP^2}(-3 + \deg(D))$. More generally, if $X = \bP^n$ then one requires $\deg D \geq n+2$. \end{example}

As we will see in the next section, there are conjectural higher dimensional analogues of Faltings' Theorem which assert hyperbolicity properties of projective varieties $X$ which are of general type. In the quasi-projective setting, there are also conjectural analogues that ask for log general type. 

\section{Lang's Conjecture}\label{sec:Lang}

We are now in a position to state the conjectural higher dimensional generalization of Faltings' Theorem. The main idea is that varieties of general type should satisfy  an analogous arithmetic behavior to curves of high genus.

The first conjecture that we mention is due to Bombieri (in the case of surfaces) and Lang: Bombieri addressed the problem of degeneracy of rational points in varieties of general type in a lecture at the University of Chicago in 1980, while Lang gave more general conjectures centered on the relationship between the distribution of rational points with hyperbolicity and Diophantine approximation (see \cite{Lang1997} and \cite{Lang1974}). The conjecture reads as follows:

\begin{conjecture}[Lang's Conjecture, Bombieri-Lang for surfaces]\label{conj:lang} Let $X$ be a (projective) variety of general type over a number field $k$. Then $X(k)$ is not Zariski dense. \end{conjecture}

We note that one cannot expect that $X(k)$ would be finite once $\dim X \geq 2$, as varieties of general type can, for example, contain rational curves, which in turn (potentially) contain infinitely many rational points.

\subsection{Generalizations of Lang and other applications}

From the previous discussion we have seen that Conjecture \ref{conj:lang} conjecturally extends Faltings' Theorem \ref{thm:faltings}. It is natural to ask whether a similar extension exists for Siegel's Theorem \ref{th:int_curves}. Indeed such a generalization exists: the role of curves with positive Euler Characteristic is played now by pairs of log general type. Then, the conjectural behavior of integral points is summarized in the following conjecture due to Vojta, and, in the following reformulation, using ideas of Lang.

\begin{conjecture}[Lang-Vojta]\label{conj:LV1}
Let $X$ be a quasi projective variety of \emph{log general type} defined over a number field $k$ and let $\calO_{S,k}$ the ring of $S$-integers for a finite set of places of $k$ containing the archimedean ones. Then the set $X(\calO_{S,k})$ is not Zariski dense.
\end{conjecture}

As was with the Lang conjecture for projective varieties, the finiteness result of Siegel becomes non-density. In higher dimensions, the positivity of the log canonical divisor is not sufficient to exclude the presence of infinitely many integral points. In particular, varieties of log general type of dimension at least two can contain (finitely many) curves that are \emph{not} of log general type.

\begin{example}\label{ex:P24} ($\PP^2$ and 4 lines). 
We can consider $D = x_0 x_1 x_2 (x_0 + x_1 + x_2)$ as a divisor in $\PP^2_\Q$, and $S$ a finite set of places. Then $(S,D)$-integral points are a subset of points of the form $[x_0 : x_1 : x_2]$ where $x_0, x_1$ and $x_2$ are $S$-units. This is equivalent to requiring that the points are integral with respect to the three lines $x_0 = 0$, $x_1 = 0$ and $x_2 = 0$. In particular we can consider points of the form $(x_0:x_1:x_2) = (1:x:y)$ with $x,y \in \calO_S^*$. The integrality with respect to the fourth line implies that $1 + x + y$ is not 0 modulo every $\fp \notin S$. So if we define $z := 1 + x + y$ then $z$ is a $S$-unit and we have that
	\[
		z - x - y = 1
	\]
	which is the classical $S$-unit equation to be solved in units. Then, as an application of Schmidt's subspace theorem \cite[Theorem 2.2.1]{vojtabook}, one gets that there are only finitely many solutions outside the three trivial families:
	\begin{align*}
 \begin{dcases}
        z =1 \\
        x = u \\
       y = -u \end{dcases}
& &  
	 \begin{dcases}
       z = u \\
        x = -1 \\
       y = u \end{dcases}
       & &
	 \begin{dcases}
       z = u \\
        y = -1 \\
       x = u \end{dcases}
       \end{align*}
 These families correspond to curves in $X$ with non positive Euler characteristic (since they intersect the divisor $D$ in only two points - passing through two of the singular points of $D$). In particular by Theorem \ref{th:int_curves}, there will be infinitely many $(S,D)$-integral points contained in these curves, up to a finite extension of $\Q$.
\end{example}

Conjecture \ref{conj:LV1} is a consequence of a more general conjecture, proposed by Paul Vojta and related to his ``landmark Ph.D. Thesis'', which gave the basis for a systematic treatment of analogies between value distribution theory and Diophantine geometry over number fields. Based on this analogy Vojta formulated a set of far-reaching conjectures. For a detailed description we refer to \cite{vojtabook} as well as chapters in the books \cite{HindrySilverman,BombieriGubler,rubook}. We will discuss this in Section \ref{sec:ff}. 

Finally we mention that more recently Campana has proposed a series of conjectures based on a functorial geometric description of varieties that aims at classify completely the arithmetic behaviour based on geometric data. For this new and exciting developments we refer to Campana's chapter \cite{Camp_book} in this book.

\subsection{Known cases of Lang's conjecture}\label{sec:knowncases}
As noted above, Faltings' second proof of the Mordell conjecture followed from his resolution of the following conjecture of Lang. 

\begin{theorem}\cite{Faltings_ab,Faltings2}\label{thm:bigfaltings}
Let $A$ be an abelian variety over a number field $K$ and let $X$ be a geometrically irreducible closed subvariety of $A$ which is not a translate of an abelian subvariety over $\overline{K}$. Then $X \cap A(K)$ is not Zariski dense in $X$. \end{theorem}

In particular, one has the following corollary.

\begin{corollary}[Faltings]\label{cor:faltings} Let $A$ be an abelian variety defined over a number field $K$. If $X$ is a subvariety of $A$ which does not contain any translates of abelian subvarieties of $A$, then $X(K)$ is finite. \end{corollary}

Using this result, Moriwaki proved the following result, whose generalization is one of the main results in these notes. 

\begin{theorem}\cite{moriwaki}\label{thm:moriwaki} Let $X$ be a projective variety defined over a number field $k$ such that $\Omega^1_X$ is ample and globally generated. Then $X(k)$ is finite. \end{theorem}

\begin{proof}[Sketch of proof] Using Faltings' Theorem \ref{thm:bigfaltings}, and the Albanese variety, one can show that if $X$ is a projective variety with $\Omega^1_X$ globally generated, then every irreducible component of $\overline{X(k)}$ is geometrically irreducible and isomorphic to an abelian variety. We will see in Proposition \ref{prop:lazarsfeld}, that if $\Omega^1_X$ is ample then all subvarieties of $X$ are of general type, and so $X$ does not contain any abelian varieties. Therefore by Corollary \ref{cor:faltings}, the set $X(k)$ is finite.
\end{proof}

For curves $\Omega^1_C \cong \omega_C$, but for higher dimensional varieties $X$, assuming positivity of the vector bundle $\Omega^1_X$ is a stronger condition than assuming positivity of $\omega_X$. In the following section, we will review positivity for vector bundles. 

\subsection{Known cases of the Lang-Vojta conjecture}
In the context of degeneracy of $S$-integral points, as predicted by Conjecture \ref{conj:LV1}, the analogue of Theorem \ref{thm:bigfaltings} is the following result due to Vojta. For the definition of semi-abelian varieties see Definition \ref{def:semiab}.

\begin{theorem}\cite{vojta1, vojta2}\label{thm:vojta}
  Let $X \subset A$ be an irreducible subvariety of a semi-abelian variety $A$ defined over a number field $k$. If $X$ is not a translate of a semi-abelian subvariety, then for every ring of $S$-integers $\calO_{k,S}$, the set of integral points $X(\calO_{k,S})$ is not Zariski dense in $X$.
\end{theorem}

\begin{corollary}\label{cor:vojta}
  In the above setting, if $X$ does not contain any translate of a semi-abelian subvariety of $A$ then $X(\calO_{D,S})$ is a finite set.
\end{corollary}

In a parallel direction, the Lang-Vojta conjecture is known when the divisor $D$ has several components: we discussed one example of this in Example \ref{ex:P24}. Such results arise from the higher dimensional extension of a method developed by Corvaja and Zannier in \cite{CZSiegel} to give a new proof of Siegel's Theorem. In \cite{CZAnnals}, Corvaja and Zannier prove a general result that implies non density of $S$-integral points on surface pairs $(X,D)$ where $D$ has at least two components that satisfy a technical condition on their intersection numbers. This result has been generalized by the same authors, Levin, Autissier et al., extending the method both to higher dimensions as well as refining the conditions on the divisor $D$;  see e.g. \cite{levin, CZ_IMRN, CZ_Toh, CZ_Adv, Autissier} -- we refer to \cite{Corvaja_book,CZ_Book} for surveys of known results.

\section{Hyperbolicity of projective varieties and ample cotangent bundles}\label{sec:hypproj}
The goal of this section is to understand the assumptions in Theorem \ref{thm:moriwaki} -- namely ampleness and global generation for vector bundles. Recall that the definition of ampleness for line bundles was given in Definition \ref{def:lbundles}. 

\begin{definition}\label{def:vbundles} Let $E$ be a vector bundle on a projective variety $X$ and let $H$ be an ample line bundle. We say that 
\begin{itemize}
\item $E$ is \emph{globally generated} if there exists a positive integer $a >0$ and a surjective map\\ $\calO_X^a \to E \to 0$. 
\item $E$ is \emph{ample} if there exists a positive integer $a > 0$ such that the sheaf $\Sym^a(E) \otimes H^{-1}$ is globally generated, and
\item $E$ is \emph{big} if there exists a positive integer $a > 0$ such that the sheaf $\Sym^a(E) \otimes H^{-1}$ is generically globally generated. \end{itemize} \end{definition}

\begin{remark} These definitions are independent of the choice of ample line bundle $H$ (see \cite[Lemma 2.14a]{vie}. \end{remark}

We note that there are alternative ways to describe ampleness and bigness for vector bundles (see \cite{hartshornevb}). 

\begin{prop} A vector bundle $E$ be on a projective variety $X$ is ample if and only if $\calO_{\bP(E)}(1)$ is ample on $\bP(E)$.\end{prop} 

\begin{remark} One can try to make the above definition for big, namely that $E$ is big if and only if $\calO_{\bP(E)}(1)$ is big, but this definition \emph{does not} always coincide with the above definition (see Example \ref{ex:big}). We will call $E$ \emph{weakly big} if $\calO_{\bP(E)}(1)$ is big to distinguish the two notions. \end{remark}

\begin{example}\label{ex:big}The vector bundle $E = \calO_{\bP^1} \oplus \calO_{\bP^1}(1)$ is weakly big, but not big as in Definition \ref{def:vbundles}. This is because any symmetric power will have a $\calO_{\bP^1}$ summand, which will become negative when tensoring with $H^{-1}$. In particular, it will never be generically globally generated. The fact that $E$ is weakly big follows from the following calculation (or see Remark \ref{rmk:pullback}). The $n$th symmetric power is $\Sym^n(E) = \calO_{\bP^2} + \calO_{\bP^1}(1) + \dots + \calO_{\bP^1}(n)$ so $h^0(\Sym^n(E)) = 1 + 2 + \dots + (n+1) = cn^2 + \dots$, and therefore grows like a degree 2 polynomial. If $X = \bP(E)$ then $X = F_1$, the Hirzebruch surface. Consider the natural map $f: F_1 \to \bP^1$ then $f_*(\calO(n)) = \Sym^n(E)$, and so $h^0(F_1, \calO(n)) = h^0(\bP^1, \Sym^n(E))$. \end{example}

We will need the following fact repeatedly:

\begin{prop}\cite[Proposition 2.2 \& 4.1]{hartshornevb}\label{prop:subbundle} Any quotient of an ample vector bundle is ample. The restriction of an ample vector bundle is ample. \end{prop}

\begin{proof}[Sketch of Proof of \ref{prop:subbundle}] We show the quotient result, the other result is similar. If $A \to B$ is a surjective map of vector bundles, then $\Sym^n(A) \otimes F \to \Sym^n(B) \otimes F$ is surjective. So if the former is globally generated, so is the latter. \end{proof}

\begin{example} Let $X = \bP^1$. Recall that any vector bundle of rank $r$ on $\bP^1$ can be decomposed as the sum of $r$ line bundles. Then $\calO_X \oplus \calO_X$ is globally generated (but not ample nor big). The vector bundle $\calO_X(1) \oplus \calO_X(1)$ is ample. \end{example}

\begin{example} Let $X = \bP^n$. Then $T_X$ is ample by the Euler sequence combined with Proposition \ref{prop:subbundle}, and $T_X(-1)$ is globally generated, but neither ample nor big. To see it is globally generated, note that tensoring the Euler sequence with $\calO_X(-1)$, we obtain

$$\calO_X(-1) \to \calO_X^{\otimes (n+1)} \to T_X(-1) \to 0.$$

The fact that it is not ample follows since the restriction of $T_X(-1)$ to a line $l \subset \bP^n$ is $\calO_l(1) \oplus \calO_l^{\otimes (n+1)}$. One can also show that $T_{\bP^n}(-1)$ is not big (see \cite[Remark 2.4]{payne}). 
\end{example}

We are now ready to prove the main result in this section.

\begin{prop}\cite[Example 6.3.28]{lazarsfeld2}\label{prop:lazarsfeld} Let $X$ be a smooth projective variety with ample cotangent bundle $\Omega^1_X$. Then all irreducible subvarieties of $X$ are of general type. \end{prop}

\begin{proof}Let $Y_0 \subset X$ be an irreducible subvariety of dimension $d$, and let $\mu: Y \to Y_0$ be a resolution of singularities. Then there is a generically surjective homomorphism $\mu^* \Omega_X^d \to \Omega^d_Y = \calO_Y(K_Y)$. Since $\Omega_X^d$ is ample, the pullback $\mu^*\Omega_X^d$ is big (see Remark \ref{rmk:pullback}) and thus $\calO_Y(K_Y)$ is also big. \end{proof}

\begin{remark}\leavevmode
\begin{enumerate}
\item It is worth noting that the converse is not true, that is there are hyperbolic varieties $X$ such that the cotangent bundle $\Omega^1_X$ is not ample, see Example \ref{ex:nonample}. 
\item In general, it is not so easy to find varieties $X$ with ample cotangent bundle $\Omega^1_X$ (see \cite[Section 6.3.B]{lazarsfeld2}).
\end{enumerate}
\end{remark}

We saw in Theorem \ref{thm:moriwaki}, that if we also assume that $\Omega^1_X$ is \emph{globally generated}, we can obtain finiteness of integral points (unconditionally with respect to Lang's conjecture). Finally, although we will not prove it, we recall that Kobayashi proved (see \cite[Theorem 6.3.26]{lazarsfeld2}) that if $X$ is a smooth projective variety with ample cotangent bundle $\Omega^1_X$, then $X$ is Brody hyperbolic (see also \cite[Proposition 3.1]{Demailly1997}). Again, there are examples of Brody hyperbolic varieties for which $\Omega^1_X$ is not ample (see Example \ref{ex:nonample}).

\begin{example}\cite[Remark 6.3.27]{lazarsfeld2}\label{ex:nonample} Let $B$ be a curve of genus $g(B) \geq 2$ and consider the variety $X = B \times B$. Then $X$ is hyperbolic but $\Omega^1_X$ is \emph{not} ample as its restriction to $Y = B \times \{\rm{pt}\}$ has a trivial quotient. On the other hand, consider a holomorphic map $\bC \to X = B \times B$. Since $B$ is hyperbolic, the map cannot be contained in a fiber. Consider the composition $\bC \to B \times B \to B$. This is a holomorphic map to a curve of genus $\geq 2$ and therefore by Liouville's theorem must be constant.  \end{example}

\begin{example}\cite[Construction 6.3.37]{lazarsfeld2}\leavevmode
\begin{enumerate}
\item Let $X_1, X_2$ be smooth projective surfaces (over $\bC$) of general type with $c_1(X_i)^2 > 2c_2(X_i)$. Then a complete intersection of two general sufficiently positive divisors in $X_1 \times X_2$ is a surface $X$ with $\Omega^1_X$ ample. 
\item Let $f: X \to B$ be a non-isotrivial family of smooth projective curves of genus $g \geq 3$ over a smooth curve of genus $g(B) \geq 2$. Then $\Omega^1_X$ is ample. These are called Kodaira surfaces.  
\end{enumerate}\end{example}

\begin{example}\cite[Example 6.3.38]{lazarsfeld2} Let $Y_1, \dots, Y_m$ be smooth projective varieties of dimension $d \geq 1$ with big cotangent bundle (e.g. if $Y_i$ are surfaces of general type with $c_1(Y)^2 > c_2(Y)$ ) and let $$X \subseteq Y_1 \times \dots \times Y_m$$ be a general complete intersection of sufficiently high multiples of an ample divisor. Then if 
$$ \dim X \leq \frac{d(m+1)+1}{2(d+1)},$$ then $X$ has ample cotangent bundle $\Omega^1_X$.
\end{example}

\begin{example}\cite{debarre} If $X$ is the complete intersection of $e \geq n$ 
sufficiently ample general divisors in a simple abelian variety of dimension
$n + e$, then the cotangent bundle $\Omega^1_X$ is ample. \end{example}

Debarre conjectured that if $X \subset \bP^r$ is the complete intersection of $e \geq r/2$ hypersurfaces of sufficiently high degree, then the cotangent bundle $\Omega^1_X$ is ample \cite{debarre}. This conjecture is now a theorem of Brotbek and Darondeau \cite{brotbekinventiones} and independently, Xie \cite{xie}. We state one of the related results. 

\begin{theorem}\cite{brotbekinventiones,xie} In every smooth projective variety $M$ for each $n \leq \dim(M)/2$, there exist smooth subvarieties of dimenson $n$ with ample cotangent bundles. \end{theorem}

\begin{remark}Brotbek-Darondeau prove Debarre's conjecture without providing effective bounds. Xie provides effective bounds, and the work of Deng \cite{deng1} improves these bounds. Work of Coskun-Riedl improves the bound in many cases \cite{coskunriedl}. 
\end{remark}

In the next section, we shift our focus to quasi-projective varieties.

\section{Hyperbolicity of log pairs}\label{sec:hyplog}
Using the ideas introduced in the above section, we now wish to understand positivity conditions on a pair $(X,D)$ that will guarantee hyperbolicity. We first define what hyperbolicity means for quasi-projective varieties. 

\begin{definition}\label{def:loghyp} Let $X$ be a projective geometrically integral variety over a field $k$, let $D$ be a normal crossings divisor on $X$, and let $V = X \setminus D$. 
\begin{itemize}
\item $V = (X \setminus D)$ is \emph{arithmetically hyperbolic} if $V(\calO_D)$ is finite. 
\item $V_\bC$ is \emph{Brody hyperbolic} if every holomorphic map $\bC \to V_\bC$ is constant.
\end{itemize} \end{definition}

Then the conjectures in the spirit of Green-Griffiths-Lang-Vojta assert that the above are equivalent, and are additionally equivalent to all subvarieties of $V$ being of log general type. Recall that aside from the canonical sheaf, the main player to study hyperbolicity was the cotangent sheaf $\Omega^1_X$. We now consider the generalization to pairs.

\begin{definition} Let $X$ be a smooth projective variety and let $D = \sum_{j=1}^r D_j$ be a reduced normal crossings divisor on $X$. The \emph{log cotangent bundle} $\Omega^1_X(\log D)$ denotes the sheaf of differential forms on $X$ with logarithmic poles along $D$. \end{definition}

For example, if $\dim X = n$ and $U \subset X$ is an open set such that $D|_U = z_1z_2\cdots z_k = 0$ (with $k < n$) then $$H^0(U, \Omega^1_X(\log D)) = \operatorname{Span}\{ \frac{dz_1}{z_1}, \dots, \frac{dz_k}{z_k}, dz_{k+1}, \dots, dz_n \}.$$

The natural idea would be to ask whether or not ampleness of the log cotangent bundle $\Omega^1_X(\log D)$ implies the desired hyperbolicity properties. It turns out that the log cotangent bundle $\Omega^1_X(\log D)$ is \emph{never} ample. Indeed, there are non-ample quotients coming from $D$ which violate the quotient property from Proposition \ref{prop:subbundle}.

\begin{prop}\cite{advt,brotbek} Let $X$ be a smooth variety of $\dim X > 1$ and $D \neq \emptyset$ a normal crossings divisor on $X$. Then the log cotangent sheaf $\Omega^1_X(\log D)$ is never ample. \end{prop}

\begin{proof} Suppose that the log cotangent bundle $\Omega^1_X(\log D)$ were ample. Consider the following exact sequence (see \cite[Proposition 2.3]{ev}):
$$0 \to \Omega^1_X \to \Omega^1_X(\log D) \to \oplus_{j = 1}^r D_j \to 0, $$
where $D_j$ are the components of $D$. Now consider the restriction of this sequence to a component $D_i \subseteq D$, and tensor with $\calO_{D_i}$. 

In this way one obtains a surjection $$A \to \calO_{D_i} + Q \to 0,$$ where $A$ is an ample sheaf (being the restriction of an ample sheaf), and $Q$ is a torsion sheaf supported at $D_i \cap D_j$, whenever $D_i \cap D_j \neq \emptyset$. We note that if there is no common point between any two irreducible components of $D$, in particular if $D$ is irreducible, then the second term of the sequence will just be $\calO_{D_i} $. Then, since $\calO_{D_i} \oplus Q$ (and in particular $\calO_{D_i}$) is \emph{not} ample, there cannot be a surjection from an ample sheaf (this would violate Proposition \ref{prop:subbundle}) and thus the log cotangent bundle $\Omega^1_X(\log D)$ cannot be ample. \end{proof}

Instead, we can ask what happens if the log cotangent bundle $\Omega^1_X(\log D)$ is, in a sense, as ample as possible. Before introducing our notion of \emph{almost ample}, we recall the definition of the augmented base locus. Let $\mathrm{Bs}(D)$ denote the base locus of $D$. 

\begin{definition}  The \emph{stable base locus} of a line bundle $L$ on a projective variety $X$ is the Zariski closed subset defined as:
$$\textbf{B}(L) := \displaystyle\bigcap_{m \in \mathbb{N}} \mathrm{Bs}(mL), $$
and the \emph{augmented base locus} (aka non-ample locus) of $L$ is:
$$\textbf{B}_+(L) := \displaystyle\bigcap_{m \in \mathbb{N}} \textbf{B}(mL -A),$$ where $A$ is any ample line bundle on $X$.  \end{definition}

\begin{remark}If $E$ is a vector bundle on $X$ we define the augmented base locus
as $\pi(\mathbf{B}_+(\calO(1)_{\bP(E)})$ where $\pi: \bP(E) \to X$. Note that $\textbf{B}_+(L)$ is empty if and only if $L$ is ample, and that $\textbf{B}_+(L) \neq X$ if and only if $L$ is big (see \cite[Example 1.7]{lazetal}. \end{remark}

\begin{example}[Nef and big divisors] Let $L$ be a big and nef divisor on $X$. We can define the \emph{null locus} $\rm{Null}(L) \subseteq X$ be the union of all positive dimensional subvarieties $V \subseteq X$ with $(L^{\dim V}\cdot V) = 0$. Then this is a proper algebraic subset of $X$ (see \cite[Lemma 10.3.6]{lazarsfeld2}), and a theorem of Nakamaye (see \cite[Theorem 10.3.5]{lazarsfeld2}) says that \[\textbf{B}_+(L)  = \rm{Null}(L).\] \end{example}

\begin{example}[Surfaces]
If $X$ is a smooth surface and $D$ is a big divisor on $X$, then there exists a \emph{Zariski decomposition} $D = P+N$, where $P$ is the \emph{nef} part, and $N$ is the \emph{negative} part. In this case, one can prove (see \cite[Example 1.11]{lazetal}) that \[\textbf{B}_+(D) = \rm{Null}(P).\]
\end{example}

\begin{theorem}\cite{baselocus} Let $\calL$ be a big line bundle on a normal projective variety $X$. The complement of the augmented base locus is the largest Zariski open set such that the morphism $\phi_m(X, L)$ is an isomorphism onto its image.\end{theorem}

Before giving our definition of almost ample, we recall the notion of augmented base loci for coherent sheaves. 

\begin{definition}\cite[Definition 2.4]{baselocus}
Let $X$ be a normal projective variety, let $\sF$ be a coherent sheaf, and $A$ an ample line bundle. Let $r = p/q \in \Q >0$. The \emph{augmented base locus} of $\sF$ is \[ \mathbf{B}_+(\sF) := \displaystyle\bigcap_{r \in \Q >0} \mathbf{B}(\sF - rA),\] where $\mathbf{B}(\sF-rA) = \mathbf{B}(\Sym^{[q]}\sF \otimes A^{-p})$.\end{definition}

\begin{remark}\leavevmode
\begin{enumerate}
    \item The augmented base locus does not depend on the choice of an ample divisor $A$. 
    \item By \cite[Proposition 3.2]{baselocus}, if $\sF$ is a coherent sheaf and  $\pi: 
\PP(\sF) = \PP(\Sym \sF) \to X$ is the canonical morphism, then $\pi(\mathbf{B}_+(\calO_{\PP(\sF)}(1))) = \mathbf{B}_+(\sF)$, i.e. the non-ample locus of $\sF$.
\end{enumerate}
\end{remark}

\begin{definition}\label{def:almostample} Let $(X,D)$ be a pair of a smooth projective variety and a normal crossings divisor $D$. We say that the log cotangent sheaf $\Omega^1_X(\log D)$ is \emph{almost ample} if 
\begin{enumerate}
\item $\Omega^1_X(\log D)$ is big, and
\item $\textbf{B}_+(\Omega^1_X(\log D)) \subseteq \rm{Supp}(D)$.
\end{enumerate}
\end{definition}

\newpage 

\begin{remark}\leavevmode
\begin{enumerate}
\item We can define the above notion more generally for singular varieties (see \cite{advt}), e.g. varieties with lc and slc singularities coming from moduli theory. This is necessary to obtain the uniformity results in loc. cit., however it is unnecessary for the proof of Theorem \ref{thm:main}. 
\item When $X$ is smooth, our notion does not quite coincide with almost ample as in \cite[Definition 2.1]{brotbek}. If the log cotangent sheaf is almost ample in the sense of \cite{brotbek}, then it is almost ample in our sense. However, our definition is a priori weaker. 
\item Brotbek-Deng proved that for any smooth projective $X$ there exists a choice of $D$ so that the log cotangent bundle $\Omega^1_X(\log D)$ is almost ample (see Theorem \ref{thm:brotbekdeng}).
\item For a log smooth pair $(X,D)$ with almost ample $\Omega^1_X(\log D)$, the complement $X \setminus D$ is Brody hyperbolic by the base locus condition (see e.g. \cite[Proposition 3.3]{Demailly1997}).
\end{enumerate} \end{remark}

We now state the above theorem of Brotbek-Deng. 

\begin{theorem}\cite[Theorem A]{brotbek}\label{thm:brotbekdeng}
Let $Y$ be a smooth projective variety of $\dim Y = n$ and let $c > n$. Let $\calL$ be a very ample line bundle on $Y$. For any $m \geq (4n)^{n+2}$ and for general hypersurfaces $H_1, \dots, H_c \in |\calL^m|$, writing $D = \sum_{i = 1}^c H_i$, the logarithmic cotangent bundle $\Omega^1_Y(\log D)$ is almost ample. In particular, $Y \setminus D$ is Brody hyperbolic.
\end{theorem}

We now present the proof that almost ample log cotangent implies that all subvarieties are of log general type. We note that the proof of the statement in full generality is outside the scope of these notes, and so we present a simplified proof which works for log smooth pairs.

 \begin{theorem}\label{thm:almostample}\cite[Theorem 1.5]{advt} Let $(X, D)$ be a log smooth pair. If $(X, D)$ has almost ample
log cotangent $\Omega^1_X(\log D)$, then all pairs $(Y, E)$ where $E := (Y \cap D)_{\textrm{red}}$ with $Y \subset X$ irreducible and not contained
in $D$ are of log general type. \end{theorem}

\begin{proof} Consider a log resolution $(\widetilde{Y}, \widetilde{E}) \to (Y, E)$, which gives a map $\phi: (\widetilde{Y}, \widetilde{E}) \to (X,D)$. Since $Y$ is not contained in $D$, by the definition of almost ample, $Y$ is not contained in the base locus of $\Omega^1_X(\log D)$. This gives a map $\phi^*(\Omega^1_X(\log D)) \to \Omega_{\widetilde{Y}}(\log \widetilde{E})$. The image of this map is a big subsheaf of $\Omega_{\widetilde{Y}}(\log \widetilde{E})$, being a quotient of a big sheaf, and thus its determinant is also big. By \cite[Theorem 4.1]{cp} (see also \cite[Theorem 1]{schnell}) $K_{\widetilde{Y}} + \widetilde{E}$ is big, and so $(Y,E)$ is of log general type. \end{proof}

\begin{remark}
In the above proof, we used that the quotient of a big sheaf is a big sheaf. We stress that this is not true for weakly big. The key idea is to be big there is a generically surjective map, and this map remains generically surjective when restricting to a subvariety not contained in the base locus (in this case a subvariety contained in the divisor $D$).   
\end{remark}

In \cite{advt}, we prove this statement in further generality. Namely, we prove the result for pairs with singularities which arise from moduli theory (i.e. lc and slc singularities). This is necessary for the proofs of uniformity in loc. cit.  We now show an alternative proof for Theorem \ref{thm:almostample} in the case $\dim X = 2$, which avoids the use of \cite{cp}. 

\begin{proof}[Alternative proof of Theorem \ref{thm:almostample} if $\dim X = 2$.] By assumption, $\Omega^1_X(\log D)$ is big and so its restriction to any subvariety $Y \not\subset D$ is still big. Since $Y$ is a curve, big is equivalent to ample, and so the restriction is actually ample. Consider the normalization $\phi: Y^v \to Y$ and denote by $E^v$ the divisor $E^v  = \phi^{-1}(E) \cup \{\textrm{ exceptional set of } \phi \}$.  Since $\Omega^1_X(\log D)|_Y$ is ample, its pullback $\phi^*(\Omega^1_X(\log D)|_Y)$ is big. There is a generically surjective map (see \cite[Theorem 4.3]{gkkp})
$$\phi^*(\Omega^1_X(\log D)|_Y) \to \Omega^1_{Y^v}(E^v) = \calO_{Y^v}(K_{Y^v} + E^v).$$ Therefore we see that $\calO_{Y^v}(K_{Y^v} + E^v)$ is big and so $(Y,E)$ is of log general type. \end{proof}

\section{Semi-abelian varieties and the quasi-Albanese map}\label{sec:semiabelian}
In this section we introduce semi-abelian varieties, and prove a generalization of Moriwaki's theorem. In particular, we show that the Lang-Vojta conjecture holds for varieties which have almost ample and globally generated log cotangent bundle. 

\begin{theorem}\cite{advt}\label{thm:main} Let V be a smooth quasi-projective variety with log smooth
compactification $(X, D)$ over a number field $K$. If the log cotangent sheaf $\Omega^1_X(\log D)$ is globally
generated and almost ample, then for any finite set of places $S$ the set of $S$-integral points $V(\calO_{K,S})$ is finite. \end{theorem}

We begin with the definition of a semi-abelian variety. Our discussion follows \cite{fujino}. 

\begin{definition}\label{def:semiab} A \emph{semi-abelian variety} is an irreducible algebraic group $A$ which, after a suitable base change, can be realized as an extension of an abelian variety by a linear torus, i.e. the middle term of an exact sequence $$1 \to \mathbb{G}_m^r \to A \to A_0 \to 1,$$ where $A_0$ is an abelian variety. \end{definition}

\begin{example}
  Immediate examples of semi-abelian varieties are tori and abelian varieties. Any product of a torus with an abelian variety is a semi-abelian variety called \emph{split}. 
\end{example}

By Vojta's generalization of Faltings' theorem (see Theorem \ref{thm:vojta} and Corollary \ref{cor:vojta}), one way to obtain finiteness of the set of integral points is to consider varieties $X \setminus D$ that satisfy the following two conditions:
\begin{enumerate}
  \item $X \setminus D$ embeds in a semi-abelian variety as a proper subscheme;
  \item $X \setminus D$ does not contain any subvariety which is isomorphic to (a translate of) a semi-abelian variety.
\end{enumerate}

Clearly the two conditions imply that the set of $D$-integral points on $X$ is finite. This strategy has some similarity with the proof of Siegel's Theorem using the Roth's Theorem, where one make use of the embedding of the curve in its Jacobian.\medskip

To embed a pair in a semi-abelian variety we will use the theory of (quasi-)Albanese maps. Recall that every variety admits a universal morphism to an abelian variety, called the \emph{Albanese map}. The same is true for quasi-projective varieties where the universal morphism instead maps to a semi-abelian variety.

\begin{definition}\cite[Definition 2.15]{fujino}\label{def:qalbanese} Let $V$ be a smooth variety. The \emph{quasi-Albanese map} $$\alpha: V \to \calA_V$$ is a morphism to a semi-abelian variety $\calA_V$ such that 
\begin{enumerate}
\item For any other morphism $\beta: V \to B$ to a semi-abelian variety $B$, there is a morphism $f: \calA_V \to B$ such that $\beta = f \circ \alpha$, and
\item the morphism $f$ is uniquely determined. \end{enumerate}
The semi-abelian variety $\calA_V$ is called the \emph{quasi-Albanese variety} of $X$ and was constructed originally by Serre in \cite[Th\'eor\`eme 7]{serre}. \end{definition}

\begin{remark}\label{rmk:semiab}\leavevmode
  \begin{itemize} \item If $V = \calC$ is a projective curve then $\calA_V$ is the abelian variety $\Jac(\calC)$.
  \item If $V = X\setminus D$ is  rational then $\calA_V$ is a torus. (Exercise)
  \item There is no semi-abelian subvariety of $\calA_V$ containing $\alpha(V)$.
\end{itemize}
\end{remark}

\subsubsection{Construction of $\calA_V$}
We briefly sketch the construction of $\calA_V$ for a smooth quasi-projective variety $V$ defined over the complex numbers. More generally if $V$ is defined over a perfect field $k$ one can define more abstractly the Albanese variety to be the dual of the Picard variety of $X$. 
In what follows we use the standard notation $q(V) = \dim H^0(X,\logct)$ and $q(X) = \dim H^0(X,\Omega^1_X)$.

 Let $\{\omega_1,\dots,\omega_{q(X)},\varphi_1,\dots,\varphi_\delta\}$ be a basis of $H^0(X,\logct)$. The quasi-Albanese variety of $V$ is $\calA_V \cong \C^{q(V)} / L$, where $L$ is the lattice defined by the periods, i.e. the integrals of the basis elements of $H^0(X,\logct)$ evaluated on a basis of the free part of $H_1(V,\Z)$. Then $\calA_V$ is a semi-abelian variety \cite[Lemma 3.8]{fujino}. If $0 \in V$ is a point of $V$, then the map $\alpha: V \to \calA_V$ is defined as
\[
  P \mapsto \Big(\int_0^P \omega_1, \dots, \int_0^P \omega_{q(X)},\int_0^P \varphi_1, \dots \int_0^P \varphi_\delta \Big).
\]
The map $\alpha$ is well defined \cite[Lemma 3.9]{fujino} and one can check that $\alpha$ is an algebraic map \cite[Lemma 3.10]{fujino}. In particular $\dim \calA_V = q(V)$. We will denote by $d(V) = \dim \alpha(V)$.

We see now that in order to use the quasi-Albanese variety we will need to impose some condition on the positivity of the sheaf $\logct$.  The main idea is that the geometric conditions on the log cotangent sheaf will ensure that we can embed $V$ inside its quasi-Albanese as a proper subvariety and then ensure that it does not contain any proper semi-abelian subvariety and conclude using Vojta's Theorem \ref{thm:vojta}.


\subsection{Proof of the main theorem}
We begin with the following. 

\begin{prop}\label{prop:tm1} 
Let $V$ be a smooth quasi-projective variety over a number field $K$. If $d(V) < q(V)$, then the closure of $V(\calO_{S,K})$ in $V$ is a proper closed subset. \end{prop}

\begin{proof}
  Assume $V(\calO_S) \neq \emptyset$. Since $d(V) < q(V)$, the quasi-Albanese map $\alpha$ is not surjective. In particular $\alpha(V)$ is a proper subvariety of a semi-abelian variety $\calA_V$. If $V(\calO_S)$ is dense in $V$, then so is its image $\alpha(V)(\calO_S)$ in $\alpha(V)$. By Vojta's Theorem (Theorem \ref{thm:vojta}), the image $\alpha(V)$ is a semi-abelian subvariety of $\calA_V$. This is a contradiction by Remark \ref{rmk:semiab} (alternatively think about $\alpha(V)(\calO_S)$ generating $\calA_V(\calO_S))$.   \end{proof}
%

Now we discuss the consequences of $\Omega^1_X (\log D)$ being almost ample and globally generated.

\begin{lemma}\label{lem:lem2} Let $V$ be a smooth quasi-projective variety with log smooth compactification $(X, D)$ over a field $k$ of characteristic zero. If the sheaf $\Omega^1_{X}(\log D)$ is almost ample and globally generated, then $q(V) \geq 2 \dim V$.\end{lemma}
\begin{proof} 
 If $P = \Proj(\Omega^1_{X}(\log D))$ and $L = \calO_P(1)$ then since $\Omega^1_{X}(\log D)$ is globally generated there is a morphism $\phi_{|L|}: P \to \PP^N$ where $\phi^*_{|L|} \calO_{\PP^N}(1) = L$ and $N = \dim_k H^0(P,L) -1$.  Furthermore, by definition $L$ is big. Then the map $\phi_{|L|}$ is generically finite which implies that 
    \[
    \dim P = \dim \phi_{\lvert L \rvert}(P) \leq N = \dim_k H^0(P,L) - 1.
  \]
    Noting that $\dim P = 2 \dim V - 1$ we obtain that
  \[
    q(V) = \dim H^0(\overline{X},\Omega^1_{X}(\log D)) \geq 2 \dim V.
  \]
   \end{proof}

\begin{theorem} Let $V \cong (X \setminus D)$ be a log smooth variety over a number field $k$, let $\calA_V$ be a semi-abelian variety and let $\alpha: V \to \calA_V$ be a morphism. If $\alpha^*(\Omega^1_{\calA_V}) \to \Omega^1_V$ is a surjective map of sheaves, then every irreducible component of $V(\calO_S)$ is geometrically irreducible and isomorphic to a semi-abelian variety. \end{theorem}

\begin{proof}
  Let $Y$ be an irreducible component of $\overline{V(\calO_S)}$. Since $Y(\calO_S)$ is dense in $Y$, we see that $Y$ is geometrically irreducible. We are thus left to show that $Y$ is isomorphic to a semi-abelian variety. For this we will use \cite[Theorem 4.2]{fujino} and so it suffices to show that $Y$ is smooth and $\alpha\vert_Y$ is \'etale.
  
  Let $B = \alpha(Y)$. Since $Y(k)$ is dense in $Y$, so is $B(k)$ in $B$. By Vojta's theorem (see Theorem \ref{thm:vojta}), $B$ is a translated of a semi-abelian subvariety of $\calA_V$. Consider the following diagram:
  
  \[
    \xymatrix{ \alpha^*(\Omega^1_{\calA_V})\vert_Y \ar[d] \ar[r] &\Omega^1_V\vert_Y \ar[d]  \\ (\alpha|_Y)^*(\Omega^1_B) \ar[r]  & \Omega^1_Y}
  \]

We know that $h: (\alpha|_Y)^*(\Omega^1_B) \to \Omega^1_Y$ is surjective. On the other hand, $\mathrm{rank}(\Omega^1_B) \leq \mathrm{rank}(\Omega^1_Y)$ and the former is locally free. Therefore $h$ is actually an isomorphism. Therefore $Y$ is smooth over $k$ and $\alpha|_Y$ is \'etale.  Thus we conclude the result by \cite[Theorem 4.2]{fujino} \end{proof}

%
%

\begin{corollary}\label{cor:mor} Let $V \cong (X \setminus D)$ be a log smooth variety over a number field $k$. If the log cotangent sheaf $\logct$ is globally generated, then for every finite set of places $S$, every irreducible component of $\overline{V(\calO_S)}$ is geometrically irreducible and isomorphic to a semi-abelian variety. \end{corollary} 

\begin{proof} Consider the quasi-Albanese map $\alpha: V \to \calA_V$. Since $\logct$ is globally generated and $H^0(V, \Omega^1_V) \otimes \calO_{\calA_V} \cong \Omega^1_{\calA_V}$ by \cite[Lemma 3.12]{fujino}  the map $\alpha^*(\Omega^1_{\calA_V}) \to \Omega^1_V$ is surjective. Therefore applying Lemma \ref{lem:lem2} gives the desired result. \end{proof}

\begin{proof}[Proof 1 of Theorem \ref{thm:main}] For a smooth variety $V$ with log smooth completion $(X,D)$, assuming that $\logct$ is almost ample implies there are no semi-abelian varieties inside $V$ (see Theorem \ref{thm:almostample}). Therefore, the set $V(\calO_S)$ is finite when $\Omega^1_X(\log D)$ is globally generated and almost ample by Corollary \ref{cor:mor}. \end{proof}

We now give a proof that does not use Theorem \ref{thm:almostample}.

\begin{proof}[Proof 2 of Theorem \ref{thm:main}]
Assume that $\overline{V(\calO_S)}$ has an irreducible component $Y$ of dimension $\dim Y \geq 1$. Let $(\overline{Y}, \overline{E})$ denote the completion of $Y$. Note that $\overline{Y}$ is geometrically irreducible. Furthermore, $\Omega^1_X(\log D)\vert_{\overline{Y}}$ is almost ample and globally generated.  Therefore $\Omega^1_{\overline{Y}}(\log \overline{E})$ is almost ample and globally generated as well. By Lemma \ref{lem:lem2}, $q(Y) \geq 2\dim Y$. Therefore, by Proposition \ref{prop:tm1},  $Y(\calO_S)$ is not dense in $Y$, which is a contradiction.
\end{proof}

\section{Vojta's Conjecture}\label{sec:vojta}
The goal of this section is to introduce Vojta's conjecture and the relevant height machinery to   present a result analogous to Theorem \ref{thm:main} in the function field setting. This result gives a height bound for integral points that is predicted by Vojta's main conjecture (Conjecture \ref{conj:Vojta}). We will see in this section that this ``main'' conjecture implies Conjecture \ref{conj:LV1}.

\subsection{Vojta's conjecture and the theory of heights}\label{sec:heights}

We will now recall the basic definition needed to state the main conjecture whose specific reformulation will imply Conjecture \ref{conj:LV1}. The main technical tool is the concept of \emph{height}, that plays a fundamental role in almost all results in Diophantine Geometry. The idea is that a height function measures the ``arithmetic complexity'' of points. In the classical case of $\PP^n$ the \emph{logarithmic height} is defined as
\[
  H(x_0:\dots:x_n) = \max_i (\lvert x_i \rvert)
\]
for a rational point $(x_0:\dots:x_n) \in \PP^n(\Q)$ with integer coordinates without common factors. Weil extended this notion to treat arbitrary height functions on algebraic varieties defined over number fields. In this language, the logarithmic height on $\PP^n$ is the height associated to a hyperplane divisor over $\Q$.

\begin{definition}[Weil's Height Machinery]\label{def:WHeight}
Let $X$ be a smooth projective algebraic variety defined over a number field $k$. There exists a (unique) map
\[
  h_{X,\_} : \Pic(X) \to \{ \text{ functions } X(\kbar) \to \R \}
\]
well-defined up to bounded functions, i.e. modulo $O(1)$, whose image $h_{X,D}$ for a class $D \in \Pic(X)$ is called a \emph{Weil height} associated to $D$. The map $h_{X,\_ }$ satisfies:
\begin{description}
\item[(a)] the map $D \mapsto h_{X,D}$ is an homomorphism modulo $O(1)$;
\item[(b)] if $X = \PP^n$ and $H \in \Pic(\PP^n)$ is the class of some hyperplane in $\PP^n$, then $h_{X,H}$ is the usual \emph{logarithmic height} in the projective space;
\item[(c)] Functoriality: for each $\kbar$-morphism $f : X \to Y$ of varieties and for each $D \in \Pic(Y)$ the following holds:
\[
h_{X, f^*D} = h_{Y, D} + O(1).
\]
\end{description}
\end{definition}

By abuse of notation, for a divisor $D$, we will denote the height corresponding to the class $\calO(D) \in \Pic(X)$ with $h_{X,D}$. The previous definition can be extended to non smooth varieties (even non irreducible ones) and over any field with a set of normalized absolute values which satisfy the product formula, see \cite{Lang1997} for further details. From the previous definition one can show the following properties for the height machinery:

\begin{proposition}[\cite{HindrySilverman},\cite{Lang1997}]\label{prop:propWHeight}
With the above notation, the function $h_{X,\_ }$ satisfies:
\begin{description}
\item[(d)] Let $D$ be an effective divisor in $X$ then, up to bounded functions, $h_{X,D} \geq O(1)$;
\item[(e)] \emph{Northcott's Theorem}: Let $A$ be an ample divisor in $X$ with associated height function $h_{X,A}$ then, for all constants $C_1$, $C_2$ and every extension $k'$ of $k$ with $[ k': k ] \leq C_2$, the following set is finite
\[
\{ P \in X(k') : h_{X,A}(P) \leq C_1\}.
\]
\end{description}
\end{proposition}

The second ingredient we need to introduce to formally state Vojta's conjecture is the notion of local height. Morally we want a function which measures the $v$-adic distance from a point to a divisor $D$ and such that a linear combination of these functions when $v$ runs over the set of places gives a Weil height for the divisor $D$. This motivates the following

\begin{definition}[Local Height]\label{def:locHeight}
Let $X$ be a smooth projective variety defined over a number field $k$. Then there exists a map
\[
\lambda_{\_} : \Pic(X) \to \{ \text{ functions } \coprod_{v \in M_k} X \setminus \supp D (k_v) \to \R \}
\]
defined up to $M_k$-bounded functions, i.e. up to constant maps $O_v(1) : M_k \to \R$ that are nonzero for finitely many places $v\in M_k$, such that:

\begin{description}
\item[(a)] $\lambda$ is additive up to $M_k$ bounded functions;
\item[(b)] given a rational function $f$ on $X$ with associated divisor $\divv(f) = D$. Then
\[
\lambda_{D,v} (P) = v(f(P)) 
\]
up to $O_v(1)$, for each $v \in M_k$ where $P \in U \subset X \setminus \supp D (k_v)$ with $U$ affine and $\max \lvert P \rvert_v = 0$ for all but finitely many $v$;
\item[(c)] Functoriality: for each $\kbar$-morphism $g : X \to Y$ of varieties and for each $D \in \Pic(Y)$ the following holds:
\[
\lambda_{g^* D,v} = \lambda_{D,v} \circ g + O_v(1);
\]
\item[(d)] if $D$ is an effective divisor then $\lambda_{D,v} \geq O_v(1)$;
\item[(e)] if $h_D$ is a Weil height for $D$ then
\[
h_D(P) = \sum_{v \in M_k} d_v \lambda_{D,v} (P) + O(1)
\]
for all $P \notin \supp D$, with $d_v = [ k_v : \Q_v ] / [ k : \Q ]$.
\end{description}
\end{definition}

For the detailed construction and related properties of local height we refer to \cite{Lang1997} and \cite{SerreMW}.
One intuition behind the work of Vojta, was the fact that local heights are arithmetic counterparts of proximity functions in value distribution theory: to see this consider a metrized line bundle $\sL$ with a section $s$ and metric $\lvert \cdot \rvert_v$: the function $P \mapsto \log \lvert s(P) \rvert_v$ is a local height at $v$. Following Vojta \cite{vojtabook} one can introduce arithmetic proximity and counting functions for algebraic varieties over number fields in the same spirit. 

\begin{definition}\label{def:mNfunct}
Let $S$ be a finite set of places of $k$, and let $(X,D)$ be a pair defined over $k$. Then the following functions are well defined:
\begin{align*}
m_{S,D} (P) &= \sum_{v \in S} d_v \lambda_{D,v}(P) \\
N_{S,D} (P) &= \sum_{v \notin S} d_v \lambda_{D,v}(P).
\end{align*}
called the \emph{arithmetic proximity function} and \emph{arithmetic counting function} respectively. By definition,
\[
h_D (P) = N_{S,D} (P) + m_{S,D} (P).
\]
\end{definition}

With these definitions we can now state the main Vojta conjecture which translates Griffiths' conjectural ``Second Main Theorem'' in  value distribution theory.

\begin{conjecture}[Vojta]\label{conj:Vojta}
Let $X$ be a smooth irreducible projective variety defined over a number field $k$ and let $S$ be a finite set of places of $k$. Let $D$ be a normal crossing divisor and $A$ an ample divisor on $X$. Then for every $\epsilon > 0$ there exists a proper closed subset $Z$ such that, for all $P \in X(k) \setminus Z$,
\begin{equation*}
m_{S,D} (P) + h_{K_X} (P) \leq \epsilon h_A (P) + O(1).
\end{equation*}
\end{conjecture}

Vojta's main conjecture \ref{conj:Vojta} is known to imply most of the open conjectures and fundamental theorems of Diophantine Geometry (Masser-Osterl\'e abc conjecture, Faltings' Theorem, \dots). 

We end this section by two propositions which show how the above stated conjectured implies the Bombieri-Lang conjecture \ref{conj:lang} and the Lang-Vojta conjecture \ref{conj:LV1}. For other implications and discussions we refer the interested reader to \cite{vojtabook} or \cite{rubook}.

\begin{remark} We recall that by Remark \ref{rmk:big} (Kodaira's Lemma), a big divisor $D$ has a positive multiple that can be written as the sum of an ample and effective divisor. In the following proofs we will always assume that this multiple is the divisor itself for simplifying the notation, as this can be done without loss of generality. \end{remark}


\begin{proposition}\label{prop:V->BL}
Vojta conjecture \ref{conj:Vojta} implies Bombieri-Lang conjecture \ref{conj:lang}.
\begin{proof}
If $X$ is of general type then $K_X$ is big, i.e. there exists a positive integer $n$ such that $nK_X = B + E$ with $B$ ample and $E$ effective, and we will assume $n=1$. Now Conjecture \ref{conj:Vojta} with $D = 0$ and $A = B$ gives
\[
(1-\epsilon) h_B(P) + h_E(P) \leq O(1).
\]
By Proposition \ref{prop:propWHeight}, $h_E(P) \geq 0$ and hence, by Northcott's Theorem \ref{prop:propWHeight}(e), the set $X(K)$ is not Zariski-dense in $X$.
\end{proof}
\end{proposition}

In order to prove that Vojta conjecture is stronger than the Lang-Vojta conjecture we need the following reformulation of the property of being $S$-integral in terms of the functions defined in Definition \ref{def:mNfunct}: a point $P$ is $S$-integral if $N_{S,D}(P) = O(1)$ and in particular $m_{S,D}(P) = h_D(P) + O(1)$. 
%
Using the characterization of bigness mentioned above (Remark \ref{rmk:big}), we prove the following.

\begin{proposition}\label{prop:V->LV}
Vojta's conjecture \ref{conj:Vojta} implies the Lang-Vojta conjecture \ref{conj:LV1}.
\begin{proof}
For a log general type variety $(X,D)$ one has
\[
K_X + D = B + E,
\]
for $B$ ample and $E$ effective. Hence Vojta's conjecture with $A = B$ gives, for $S$-integral points,
\[
(1 -\epsilon) h_B (P) + h_E(P) \leq O(1).
\]
As before, $h_E(P) \geq 0$; thus, using Northcott's Theorem, the set of $S$-integral points of $(X,D)$ is not Zariski dense.
\end{proof}
\end{proposition}

%
%
%
%
%
%
%
%

%

\section{Function Fields}\label{sec:ff}

Function fields in one variable and number fields share several properties. This deep analogy was observed in the second half of the 19th century; one of the first systematic treatments can be found in the famous paper by Dedekind and Weber \cite{Dedekind}. Further descriptions, due to Kronecker, Weil and van der Waerden, settled this profound connection which finally became formally completed with the scheme theory developed by Grothendieck.

\begin{definition}[Function Field]\label{def:FunctField}
A \emph{function field} $F$ over an algebraically closed field $k$ is a finitely generated field extension of finite transcendence degree over $k$. A function field \emph{in one variable}, or equivalently a function field of an algebraic curve, is a function field with transcendence degree equal to one
\end{definition}

\begin{remark}
With the language of schemes the function field of a curve $X$, or more general of every integral scheme over an algebraic closed field, can be recovered from the structure sheaf $\calO_X$ in the following way: given any affine open subset of $X$, the function field of $X$ is the fraction field of $\calO_X(V)$. Moreover, if $\eta$ is the (unique) generic point of $X$, then the function field of $X$ is also isomorphic to the stalk $\calO_{X,\eta}$.
\end{remark}

The analogy between number fields and function fields of curves, also known as algebraic function fields in one variable, comes from the fact that one-dimensional affine integral regular schemes are either smooth affine curves over a field $k$ or an open subset of the spectrum of the ring of integers of a number field. Formally, given a number field $k$ with ring of integers $\calO_k$ the scheme $\Spec\calO_k$ is one dimensional affine and integral. From this analogy, several classical properties of number fields find an analogue in the theory of function field. In particular the theory of  heights can be defined over function fields.

\begin{definition}\label{def:FFabs}
Given a function field $F$ in one variable of a non singular curve $\calC$, each (geometric) point $P \in \calC$ determines a non trivial absolute value by
\[
\lvert f \rvert_{P}:= e^{-\ord_P(f)}.
\]
Moreover if $Q \neq P$ then the absolute values $\lvert \cdot \rvert_{Q}$ and $\lvert \cdot \rvert_{P}$ are not equivalent.
\end{definition}

\begin{remark}\leavevmode
\begin{itemize}
\item The definition could have been given more generally for function fields of algebraic varieties regular in codimension one (or rather for regular models of higher dimensional function fields), replacing the point $P$ with prime divisors. Extensions exist also for function fields over non-algebraically closed fields in which one should replace points with orbits under the absolute Galois group.
\item From the fact that any rational function $f$ on a projective curve has an associated divisor of degree zero, it follows that the set of absolute values satisfy the product formula.
\end{itemize}
\end{remark}

Given the set of absolute values $M_F$ for a function field in one variable $F$, normalized in such a way that they satisfy the product formula, heights can be defined for $F$ in the following way:

\begin{definition}\label{def:heightFF}
Let $F=k(\calC)$ be as before. For any $f \in F$ the \emph{height} of $f$ is
\[
h(f) = - \sum_{P \in \calC} \min \{ 0, \ord_P(f)\} = \sum_{P \in \calC} \max \{ 0, \ord_P(f)\}.
\]
In the same way for a point $g \in \PP^n(F)$, $g = (f_0 : \dots : f_n)$, its height is defined as
\[
h(g) = - \sum_{P \in \calC} \min_i \{ \ord_P(f_i) \}.
\]
\end{definition}

From the definition it follows that a rational function on a regular curve has no poles if and only if its height is zero if and only if it is constant on the curve.

We end this subsection with Table \ref{tab:NFFF}, which illustrates the interplay and the similarity between number fields and function fields. We stress in particular how each geometric object in the right column, in particular dominant maps and pullbacks, are analogous to purely arithmetic notions like extensions of fields and extensions of ideals. This analogy can be further explored using Arakelov Theory and extending the notion of divisors to number fields by suitably compactifying the affine curve $\Spec \calO_{k,S}$; in this framework an intersection theory can be defined for such compactified divisors sharing many analogous properties of intersection theory on the geometric side. We refer to \cite{Lang1988} for further details on this subject. 

\begin{table}[h!]
\centering
\begin{tabular}{|c|c|}
\hline 
Number Field & Function Field \\ 
\hline 
$\Z$ & $k [ x ]$	 \\
\hline 
$\Q$ & $k (x)$ \\ 
\hline 
$\Q_p$ & $k((x))$ \\ 
\hline 
$k$ finite extension of $\Q$ & $F$ function field of $\calC$ \\ 
\hline 
place & geometric point \\ 
\hline 
finite set of places & finite set of points \\ 
\hline 
ring of $S$-integers & ring of regular functions\\ 
\hline 
$\Spec \calO_{k,S}$ & Affine curve $\calC \setminus S$  \\ 
\hline 
product formula & $\deg$ principal divisor = 0 \\ 
\hline 
extension of number fields & dominant maps \\ 
\hline 
extension of ideals & pullback of divisors \\
\hline 
\end{tabular}
\caption{Number Fields and Function Fields analogy}\label{tab:NFFF} 
\end{table}
\subsection{Mordell Conjecture for function fields}\label{subsec:mordell_ff}
Over function fields one cannot expect Faltings' Theorem \ref{thm:faltings} to hold as shown by the following examples.

\begin{example}
Let $\calC$ be a curve with $g(\calC) > 1$  defined over $\C$ and consider the trivial family $\calC \times \PP^1 \to \PP^1$. The family can be viewed as the curve $\calC$ (trivially) defined over the function field $\C( t)$ of $\PP^1$. All fibers of the family, being isomorphic to $\calC$ have genus greater than one. The Mordell Conjecture over function fields, without any other restriction, should imply that the set of $\C( t )$-rational points of $\calC$, i.e. there are finitely many sections $\PP^1 \to \calC \times \PP^1$. However this is easily seen to be false by considering constant sections $\PP^1 \to \{ P \} \times \PP^1$ for each point $P \in \calC(\C)$. In particular, the general type curve $\calC$ defined over $\C( t )$ has infinitely many $\C( t )$-rational points.
\end{example}

From the previous example one could guess that the problem relied on the fact that the family was a product and the curve $\calC$ was actually defined over the base field $\C$ rather than on the function field $\C ( t )$, i.e. the family was trivial. However, as the following example shows, things can go wrong even for non-trivial families.

\begin{example}[Gasbarri \cite{Gasbarrinotes}]
Consider the curve $\calC := (x + t y)^4 + y^4 -1$ defined over $\C( t )$. It has an associated fibration $\calC \to \PP^1$ whose generic fiber $\calC_{t_0} = (x + t_0 y)^4 + y^4 -1$ is a smooth projective curve of genus $g(\calC_{t_0}) = 3$. Again if we consider the same statement of Theorem \ref{thm:faltings} only replacing the number field with the function field $\C( t )$ we would expect that the number of $\C( t )$-rational points of $\calC$ to be finite. However we claim that $\calC(\C( t ))$ is infinite; to see this consider the equation $\alpha^4 + \beta^4 = 1$ over $\C^2$: it has infinitely many solutions. Each solution gives a $\C$-point of $\calC_{t_0}$, namely $(\alpha - t_0 \beta, \beta)$ proving the claim. Moreover the family is not trivial in the sense of the previous example, i.e. $\calC$ is not defined over $\C$. Notice however that each fiber of the family is isomorphic to the curve $x^4 + y^4 = 1$ via $x+t y \mapsto x$ and $y \mapsto y$.
\end{example}

Motivated by the previous examples we give the following:

\begin{definition}\label{def:isotrivial}
Given a family of irreducible, smooth projective curves $\calC \to \calB$ over a smooth base $\calB$, we say that the family is \emph{isotrivial} if $\calC_b$ is isomorphic to a fixed curve $\calC_0$ for $b$ in an open dense subset of $\calB$. With abuse of notation, we will say that a curve $\calC$ defined over a function field $F$ is \emph{isotrivial} if the corresponding fibration $\calC \to \calB$ is isotrivial, where $\calB$ is a curve with function field $F$.
\end{definition}

Isotriviality extends the notion of (birational) triviality for family of curves, i.e. a product of curves fibered over one of the factors is immediately isotrivial. At the same time this notion encompasses many other families that are not products, like the one defined in the previous example. However, after a cover of the base of the family, each isotrivial family becomes trivial; in particular the following easy lemma holds:

\begin{lemma}\label{lem:isotr->tr}
Given a isotrivial family $\calC \to \calB$ of smooth projective irreducible curves, there exists a cover $\calB' \to \calB$ such that the base changed family $\calC \times_\calB \calB' \to \calB'$ is a generically trivial family, i.e. is birational to a product $\calC \times_\C \calB_0$.
\end{lemma}

Lemma \ref{lem:isotr->tr} implies that rational points for curves defined over function fields will not be finite for isotrivial curves. The analogous form of Mordell Conjecture for function fields thus asks whether this holds only for this class of curves. We can then restate Theorem \ref{thm:faltings} in the following way:

\begin{theorem}\label{thm:geo_mordell}
Let $\calC$ be a smooth projective curve defined over a function field $F$ of genus $g(\calC)>1$. If $\calC(F)$ is infinite then $\calC$ is isotrivial.
\end{theorem}

Theorem \ref{thm:geo_mordell} was proved in the sixties by Manin \cite{Manin1963} (although with a gap fixed by Coleman \cite{Coleman1990}) using analytic arguments, and later by Grauert \cite{Grauert1965} using algebraic methods. Samuel in \cite{Samuel1966} gave a proof in characteristic $p$ using ideas of Grauert. A detailed explanation of Grauert methods can be found in Samuel's survey \cite{SamuelTata}. In Mazur's detailed discussion of Faltings' proof of Mordell Conjecture \cite{Mazur1987}, Mazur stresses the role of Arakelov \cite{Arakelov1971} and Zahrin's \cite{Zarhin1974} results which imply new proofs of the Geometric Mordell Conejcture, using ideas of Parshin: this gives even more importance to the geometric case.

One of the ideas of Grauert's proof, which is central in some of the higher dimensional extensions is the following: suppose $\calC$ is a curve defined over a function field $F$ of a curve $\calB$, corresponding to a fibration $\pi: X \to \calB$. Then one can prove that almost all sections of the fibration, which correspond to rational points, verify a first order differential equation, i.e. almost all sections are tangent to a given horizontal vector field. Formally each section $\sigma: \calB \to X$ can be lifted to the projective bundle $\calB \to \PP(\Omega^1_X) = \Proj(\Sym(\Omega_X^1))$ via the surjective map $\sigma^*\Omega^1_X \to \Omega^1_\calB$. Grauert proves (in a different language) that there exists a section $\phi$ of a suitable line bundle over $\PP(\Omega^1)$ whose zero section contains all but finitely many images of sections. Grauert then concludes that if infinitely many sections exist, given the fact that they satisfy the differential equation given by $\phi = 0$, a splitting is provided for the relative tangent sequence which implies that the family is isotrivial (via the vanishing of the Kodaira-Spencer class).

In particular, Grauert's construction gives first insights towards the theory of jet spaces which occupy a central role in some degeneracy results in the complex analytic setting. In this direction, recent analogues of Theorem \ref{thm:geo_mordell} in higher dimension have been proved by Mourougane \cite{Mourougane} for very general hypersurfaces in the projective space of high enough degree using proper extensions of the ideas briefly described above.

\subsection{Vojta Conjecture for Function Fields}\label{ssec:vojta_ff}
Since function fields possess a theory of heights analogous to the theory over number fields, one can translate Vojta's Main Conejcture \ref{conj:Vojta} to the function field case. The main conjecture implies the following height bound for varieties of log general type type over function fields.

\begin{conjecture}\label{conj:vojta_ff}
Let $(\calX,\calD)$ be a pair over a function field $F = k(B)$ whose generic fiber $(X,D)$ is a pair of log general type. Then, for every $\varepsilon > 0$ there exists a constant $C$ and a proper closed subvariety $Z$ such that for all $P (\in \calX \setminus Z) (\overline{F})$ one has
\begin{equation}
    \label{eq:vojta_ff}
    h_{K_\calX + D}(P) \leq C (\chi(P) + N^{(1)}_D(P)) +O(1)
\end{equation}
where, given a point $P \in \calX(L)$ corresponding to a cover $B_P \to B$ of degree $n$, corresponding to the field extension $L \supset F$, we have that $\chi(P) = \chi(B_P)/n$. Moreover, the truncated counting function $N^{(1)}_D (P)$ is the cardinality of the support of $P^*D$.
\end{conjecture}

Note that for varieties of log general type the height in \eqref{eq:vojta_ff} is associated to a big divisor. In this case, if the set of points of bounded height is Zariski dense then the model is isotrivial.
Moreover, if one considers only points defined over $F$ then the characteristic of the point $P$ reduces to $2(g(B)) - 2$ and one recovers the usual conjecture for $(D,S)$-integral points where $\# S \geq N^{(1)}_D(P)$.

In this latter case one can related Conjecture \ref{conj:vojta_ff} to hyperbolicity using the following result of Demailly.


\begin{theorem}[Demailly \cite{Demailly1997}]\label{th:AlgHypDemailly}
Let $X$ be a projective complex variety embedded in some projective space for a choice of a very ample line bundle. Then if the associated manifold is Kobayashi hyperbolic the following holds: there exists a constant $A>0$ such that each irreducible curve $\calC \subset X$ satisfies
\[
\deg \calC \leq A (2 g(\tcalC) - 2) = A \chi(\tcalC),
\]
where $\tcalC$ is the normalization of $\calC$.
\end{theorem}

Following this result, Demailly introduced the following notion.

\begin{definition}\label{def:alghyp}
A smooth projective variety $X$ is \emph{algebraically hyperbolic} if there exists a constant $A$ such that for each irreducible curve $\calC \subset X$	the following holds:
\[
\deg \calC \leq A \chi(\tcalC).
\]
\end{definition}

Using strong analogies between hyperbolicity and degeneracy of rational points Lang conjectured that a general type variety should be hyperbolic outside a proper exceptional set and therefore one could also conjecture that the variety should be algebraically hyperbolic outside that set (for more on algebraically hyperbolic varieties we refer to \cite{jav,jav2}). This allows one to rephrase Conjeture \ref{conj:vojta_ff} as follows.

\begin{conjecture}[Lang-Vojta for function fields]\label{conj:alghyp}
Given an affine variety $X$ embedded as $\overline{X} \setminus D$ for a smooth projective variety $\overline{X}$ and a normal crossing divisor $D$, if $X$ is of log general type then there exists a proper subvariety $Exc$ (called the \emph{exceptional set}) such that there exists a bound for the degree of images of non-constant morphisms $\calC \to X$ from affine curves whose image is not entire contained in $Exc$, in terms of the Euler Characteristic of $\calC$.
\end{conjecture}

By the previous remark it is easy to see that Conjecture \ref{conj:vojta_ff} implies Conjecture \ref{conj:alghyp}.

We note that most of the known techniques used for the number field case can be used to prove analogous results in the function field setting. However, due to the presence of tools that are not available over number fields, most notably the presence of derivation, one can obtained stronger results that lead to cases of Conjecture \ref{conj:alghyp} and Conjecture \ref{conj:vojta_ff} in settings that are currently out of reach in the function function field case. We refer to the articles \cite{Voloch,BM,Chen,Wang,PR,CZConic,Mourougane,CZGm,Amos,CaTur} as some examples of results over function fields along these lines.

\begin{remark} For the sake of completion, we discuss briefly how \emph{algebraic hyperbolicity} fits in with our previous discussions on hyperbolicity (see \cite[Example 6.3.24]{lazarsfeld2}.
\begin{itemize}
    \item If $X$ is algebraically hyperbolic, then $X$ contains no rational or elliptic curves.
    \item If $X$ is algebraically hyperbolic then there are no non-constant maps $f: A \to X$ from an abelian variety $A$.
    \item Kobayashi (and thus Brody) hyperbolicity implies algebraic hyperbolicity for projective varieties.
\end{itemize}
Furthermore, a theorem of Kobayashi (see \cite[Theorem 6.3.26]{lazarsfeld2}) states that if $\Omega^1_X$ is ample, then $X$ is algebraically hyperbolic. 
\end{remark}

\subsection{Moriwaki for Function Fields}
The analogue of Theorem \ref{thm:moriwaki} over function fields is the following theorem due to Noguchi.

\begin{theorem}[Noguchi \cite{noguchi}]\label{th:noguchi}
Let $X$ be a smooth variety over a function field $F$. If $\Omega^1_X$ is ample then the conclusion of Conjecture \ref{conj:vojta_ff} holds.
\end{theorem}

It is therefore natural to consider the analogous question for pairs. As pointed out several times in these notes, for a pair $(X,D)$, the analogous assumption on the positivity of the log cotangent, is to require that $\Omega^1_X(\log D)$ is almost ample. In this setting the following was suggested to us by Carlo Gasbarri.

\begin{Expectation}
  \label{th:Carlo}
Let $(X_F,D)$ be a log smooth non isotrivial pair over $F$. If $\Omega^1_{X_F/F}(\log D)$ is almost ample then there exists a constant $A$ and a proper closed subset $Z \subsetneq X_F$ such that for every $p \in (X_F \setminus Z)(\overline{F})$ we have that
\begin{equation*}
    h_{K_\calX + D}(P) \leq C (\chi(P) + N^{(1)}_D(P)) +O(1)
\end{equation*}
where $P$ is a model of $p$ over $\calC$.
\end{Expectation}

The intuition is as follows: first one obtains a height bound for lifts of sections over the projectivization of the model of the log cotangent sheaf. Then using the almost ample hypothesis together with the non-isotriviality of the pair, one shows that the base locus of the structure sheaf of the projectivized bundle does not dominate the base.

\section{Consequences of Lang's conjecture}\label{sec:consequences}
For the sake of completeness, and due to our personal interests, we conclude these notes with a few consequences of Lang's conjecture.

\subsection{Consequences of Lang's conjecture -- uniformity}\label{sec:uniformity}
Caporaso-Harris-Mazur \cite{CHM} showed that Conjecture \ref{conj:lang} implies that $\#\calC(K)$ in Faltings' Theorem is not only finite, but is also uniformly bounded by a constant $N=N(g,K)$ that does \emph{not} depend on the curve $\calC$.

\begin{theorem}[see \cite{CHM}]
\label{conj:unif}
  Let $K$ be a number field and $g \geq 2$ an integer. Assume Lang's conjecture. Then there exists a number $B=B(K,g)$ such that for any smooth curve $\calC$ defined over $K$ of genus $g$ the following holds: $ \# \calC(K) \leq B(g,K) $
\end{theorem}

 Pacelli \cite{Pacelli} (see also \cite{Aquadratic}), proved that $N$ only depends on $g$ and $[K:\Q]$. Some cases of Theorem \ref{conj:unif} have been proven unconditionally (\cite{krz}, \cite{stoll} and \cite{paz}) depending on the Mordell-Weil rank of the Jacobian of the curve and for \cite{paz}, on an assumption related to the Height conjecture of Lang-Silverman. It has also been shown that families of curves of high genus with a uniformly bounded number of rational points in each fiber exist \cite{dnp}. 

N\"aive translations of uniformity fail in higher dimensions as subvarieties can contain infinitely many rational points. However, one can expect that after removing such subvarieties the number of rational points is bounded. Hassett proved that for surfaces of general type this follows from Conjecture \ref{conj:lang}, and that the set of rational points on surfaces of general type lie in a subscheme of uniformly bounded degree \cite{Hassett}.

The main tool used to prove the above uniformity results is the \emph{fibered power theorem} and was shown for curves in \cite{CHM}, for surfaces \cite{Hassett} and in general by Abramovich \cite{Afpt}. In higher dimensions, similar uniformity statements hold conditionally on Lang's conjecture, and follow from the fibered power theorem under some additional hypotheses that take care of the presence of subvarieties that are not of general type (\cite{AV}).

\subsubsection{Consequences of the Lang-Vojta conjecture -- uniformity}
We saw above that Lang's conjecture had far reaching implications for uniformity results on rational points for varieties of general type. One can analogously ask if the Lang-Vojta conjecture implies uniformity results for integral points on varieties of log general type. It turns out that such results are much more subtle in the pairs case, but we review some of the known results here. 

This question was first addressed in \cite{Aell} when Abramovich asked if the Lang-Vojta conjecture  implies uniformity statements for integral points. Abramovich showed this cannot hold unless one restricts the possible models used to define integral points (see Example \ref{ex:blowup}). Instead, Abramovich defined \emph{stably integral points}, and proved uniformity results (conditional on the Lang-Vojta conjecture), for stably integral points on elliptic curves, and together with Matsuki \cite{AM} for principally polarized abelian varieties. While we do not give a precise definition of stably integral points in these notes, we remark that they are roughly integral points which remain integral after stable reduction. We refer the interested reader to our paper \cite{advt}. 

In \cite{advt}, we prove various generalizations of the work of Abramovich and Abramovich-Matsuki. In particular, we prove that the Lang-Vojta conjecture implies that the set of stably integral points on curves of log general type is uniformly bounded. Additionally, we prove a generalization of Hassett's result, showing that the Lang-Vojta conjecture implies that (stably) integral points on families of log canonically polarized surfaces lie in a subscheme whose degree is uniformly bounded. Finally, we prove, assuming the Lang-Vojta conjecture, and under the assumption that the surfaces have almost ample log cotangent, that the set of stably integral points on polarized surfaces is uniformly bounded. 

  Finally, we note that results present in \cite{advt} have two key ingredients. One is a generalization of the fibered power theorems mentioned in Section \ref{sec:uniformity} to the case of pairs \cite{fpt}. The other, is a generalization of Theorem \ref{thm:almostample}, which gives a condition so that subvarieties of a \emph{singular} surface of log general type are curves of log general type. It turns out that proving a result for stably integral points requires the use of the compact moduli space of stable pairs, and as such, we are forced to work with singular surfaces. 

\subsection{Consequences of Lang's conjecture -- rational distance sets}
A \emph{rational distance set} is a subset $S$ of $\R^2$ such that the distance between any two points of $S$ is a rational number. In 1946, Ulam asked if there exists a rational distance set that is dense for the Euclidean topology of $\R^2$.  While this problem is still open, Shaffaf \cite{shaffaf} and Tao \cite{tao} independently showed that Lang's conjecture implies that the answer to the Erd\H os-Ulam question is `no'. In fact, they showed that if Lang's conjecture holds, a rational distance set cannot even be dense for the Zariski topology of $\R^2$, i.e. must be contained in a union of real algebraic curves. 

Solymosi and de Zeeuw \cite{sz} proved (unconditionally, using Faltings' proof of Mordell's conjecture) that a rational distance contained in a real algebraic curve must be finite, unless the curve has a component which is either a line or a circle.
Furthermore, any line (resp. circle) containing infinitely many points of a rational distance set must contain all but at most four (resp. three) points of the set.  One can rephrase the result of \cite{sz} by saying that almost all points of an infinite rational distance set contained in a union of curves tend to concentrate on a line or circle. It is therefore natural to consider the ``generic situation'', and so we say that a subset $S\subseteq \R^2$ is in \emph{general position} if no line contains all but at most four points of $S$, and no circle contains all but at most three points of $S$. For example, a set of seven points in $\R^2$ is in general position if and only if no line passes through $7-4 = 3$ of the points and no circle passes through  $7-3 =  4$ of the points.

In particular, the aforementioned results show that Lang's conjecture implies that rational distance sets in general position must be finite. With Braune, we proved the following result. 

\begin{theorem}\cite[Theorem 1.1]{erdosulam}
Assume Lang's conjecture. There exists a uniform bound on the
 the cardinality of a rational distance set in general position.
\end{theorem}

We can rephrase this theorem as follows. 

\begin{corollary}\label{cor:Lang}\cite[Corollary 1.2]{erdosulam}
If there exist rational distance sets in general position of cardinality larger than any fixed constant, then Lang's conjecture does not hold.
\end{corollary}
We note that we are unaware of any examples of rational distance sets in general position of cardinality larger than seven (the case of seven answered a question of Erd\H{o}s, see \cite{kk})

\bibliographystyle{alpha}	
\bibliography{montreal}

\end{document}